\documentclass{amsart}

\usepackage{amsfonts}
\usepackage{graphicx}
\usepackage{amsmath}
\usepackage{amssymb}
\usepackage[english]{babel}
\usepackage{tikz-cd}
\usepackage[hidelinks]{hyperref}

\newcommand{\Set}{\mathsf{Set}}
\newcommand{\Vect}{\mathsf{Vect}}
\newcommand{\Grp}{\mathsf{Grp}}
\newcommand{\HopfAlg}{\mathsf{HopfAlg}}
\newcommand{\VectK}{\Vect_K}

\newtheorem{theorem}{Theorem}[section]

\newtheorem{proposition}{Proposition}[section]

\newtheorem{definition}{Definition}[section]

\newtheorem{example}{Example}[section]

\begin{document}

\title{The Grouplike Functor and Monoidal Adjunctions}

\author{Frank Murphy-Hernandez}
\address{Facultad de Ciencias, UNAM, Mexico City}
\email{murphy@ciencias.unam.mx}


\date{\today}

\begin{abstract}
We construct a categorification of the functor of grouplike elements of a Hopf algebra in the setting of braided monoidal categories. Given a braided bistrong monoidal functor admitting a right adjoint, we define a functor assigning to each Hopf monoid a group object obtained as a subobject of its underlying object by means of finite limits. We prove that this construction is functorial and provides a right adjoint to the induced free Hopf monoid functor, thereby extending the classical correspondence between groups and grouplike elements of Hopf algebras to a categorical setting.
\end{abstract}

\maketitle

%%%%%%%%%%%%%%%%%%%%%%%%%%%%%%%%%%%%%%%%%%%%%%%%%%%%%%%%%%%%%%%%%%%%
%%%%%%%%%%%%%%%%%%%%%%%%%%%%%%%%%%%%%%%%%%%%%%%%%%%%%%%%%%%%%%%%%%%%
\section*{Introduction}

The theory of Hopf algebras occupies a central position in modern algebra, with deep connections to algebraic groups, quantum groups, representation theory, and combinatorics. One of the most elementary yet conceptually rich constructions associated with a Hopf algebra is the set of its grouplike elements, namely those elements $x$ satisfying
\[
\Delta(x)=x\otimes x,\qquad \varepsilon(x)=1.
\]
These elements form a group under the multiplication of the Hopf algebra, establishing a natural bridge between Hopf algebra theory and group theory. This construction defines the classical functor
\[
G:\HopfAlg_K\longrightarrow\Grp,
\]
which is right adjoint to the group algebra functor
\[
K[-]:\Grp\longrightarrow\HopfAlg_K.
\]
Grouplike elements play a fundamental role throughout Hopf algebra theory, appearing in the study of representations, algebraic groups, quantum groups, and combinatorial Hopf algebras, where they encode the invertible symmetries of the underlying coalgebraic structure \cite{Sweedler,Montgomery,Abe}.

A categorical framework for studying Hopf algebras beyond the classical setting of vector spaces was developed by Porst in \cite{FT}, where Hopf monoids were introduced in arbitrary braided monoidal categories. One of the fundamental observations of this theory is that Hopf monoids provide the appropriate categorical analogue of groups: indeed, Hopf monoids in the cartesian monoidal category of sets are precisely groups. This viewpoint explains many familiar properties of Hopf algebras as manifestations of more general categorical phenomena and provides a natural setting in which classical constructions can be reconsidered from a broader perspective.

From the classical point of view, the adjunction
\[
K[-]\dashv G
\]
is usually regarded as a specific feature of Hopf algebras over a field. A closer inspection, however, reveals that this adjunction originates from the free--forgetful adjunction
\[
K^{(-)}:\Set\rightleftarrows\VectK:U,
\]
whose left adjoint is naturally bistrong monoidal. Since this adjunction lifts to the corresponding categories of Hopf monoids, and since
\[
\mathsf{Hopf}(\Set)\cong\Grp,
\]
the classical group algebra--grouplike elements adjunction appears as the special case of a more general categorical construction. This observation naturally raises the question of whether such a construction exists for arbitrary braided monoidal categories.

One of the guiding principles of modern category theory is \emph{categorification}: the systematic replacement of algebraic structures by their categorical analogues. Since the pioneering work of Crane and Frenkel \cite{CraneFrenkel}, categorification has become a powerful tool in representation theory, topology, and quantum algebra, revealing conceptual relationships that remain hidden at the level of sets. Motivated by the preceding observation, we ask whether the classical grouplike elements functor admits a categorical analogue arising from an arbitrary bistrong monoidal adjunction.

The purpose of this paper is to answer this question affirmatively. Given a braided bistrong monoidal functor admitting a right adjoint, we construct a functor assigning to every Hopf monoid a group object obtained as a finite-limit subobject of its underlying object. We prove that this assignment is functorial and that it is right adjoint to the induced free Hopf monoid functor. Consequently, the classical adjunction between groups and Hopf algebras is recovered as the particular case arising from the free--forgetful adjunction between sets and vector spaces. Our main result (Proposition~\ref{prop:grouplike_adjunction}) therefore shows that the group algebra--grouplike elements correspondence is an instance of a general categorical principle rather than a phenomenon specific to Hopf algebras over a field.

The principal contributions of this paper are the following.
\begin{itemize}
    \item We identify the classical group algebra--grouplike elements adjunction as the lift of the free--forgetful adjunction between sets and vector spaces to the corresponding categories of Hopf monoids.
    \item We show that this phenomenon is not specific to vector spaces, but holds for arbitrary braided bistrong monoidal adjunctions satisfying suitable hypotheses.
    \item We construct the corresponding grouplike object functor and prove that it is right adjoint to the induced free Hopf monoid functor.
    \item We recover the classical adjunction between groups and Hopf algebras as a special case of the general theory.
\end{itemize}

The paper is organised as follows. In Section~\ref{sec:braided} we recall the basic notions of monoidal and braided categories, establishing the categorical framework used throughout the paper. Section~\ref{sec:monoids} introduces monoids internal to a monoidal category, and Section~\ref{sec:bimonoids} introduces their dual notion of comonoids together with bimonoids. Section~\ref{sec:hopf_monoids_final} then introduces Hopf monoids, the central objects of our study. Section~\ref{sec:monoidal_functors} develops the theory of lax, colax, bilax, and bistrong monoidal functors, culminating in Kelly's doctrinal adjunction theorem \cite{DA}, which provides the mechanism for transporting monoidal structures across adjunctions. In Section~\ref{sec:grouplike} we construct the grouplike object functor and establish the main adjunction theorem. Finally, Section~\ref{sec:examples} illustrates the theory through several examples, including the classical adjunction between groups and Hopf algebras, Hopf monoids in monoidal model categories, and Hopf monoids in the category of species.

Throughout the paper we assume familiarity with basic category theory, including limits, adjunctions, and monoidal categories. For the classical theory of Hopf algebras we refer to \cite{Sweedler,Abe}, while the monoidal aspects of Hopf monoids are treated comprehensively in the monograph of Aguiar and Mahajan \cite{MF} and in the standard categorical references \cite{borceux1994handbook}.

%%%%%%%%%%%%%%%%%%%%%%%%%%%%%%%%%%%%%%%%%%%%%%%%%%%%%%%%%%%%%%%%%%%%
\section{Preliminaries}\label{sec:preliminaries}

Throughout this paper, $K$ denotes a field. We begin by fixing the categorical framework that will serve as our ambient setting. Let $\VectK$ denote the category of $K$-vector spaces and $K$-linear transformations; for a $K$-vector space $V$ and a set $X$, we write $V^{(X)} := \bigoplus_{x\in X} V$ for the direct sum of $|X|$ copies of $V$. The category $\VectK$ is symmetric monoidal under the usual tensor product $\otimes_K$, with unit object $K$. For the general theory of categories and their basic constructions, we refer the reader to \cite{borceux1994handbook}.

In addition to vector spaces, we shall make use of two other standard categories. We denote by $\Set$ the category of sets and functions, which is cartesian monoidal with the cartesian product as tensor product and the singleton set $\{ * \}$ as unit object. We also denote by $\Grp$ the category of groups and group homomorphisms.

Finally, we turn to the main algebraic structures of interest. Let $\HopfAlg_K$ denote the category of $K$-Hopf algebras and Hopf algebra morphisms. For the classical theory of Hopf algebras, we follow the standard treatments of \cite{Sweedler} and \cite{Abe}; for the monoidal and functorial aspects that underlie our categorical approach, we also refer to the comprehensive monograph \cite{MF}. We recall that a Hopf algebra is both a unital associative algebra and a counital coassociative coalgebra, with the additional structure of an antipode, all compatible in the usual sense.

The categorical generalisations of these structures will be developed in the sections that follow. In particular, we shall work in an arbitrary monoidal category $(\mathcal{C}, \otimes, \mathbf{1})$, where the associativity and unit constraints are understood to be coherent via the triangle and pentagon identities. When the monoidal structure is braided, we denote the braiding by $\beta_{X,Y} \colon X \otimes Y \to Y \otimes X$, satisfying the hexagon identities.
%%%%%%%%%%%%%%%%%%%%%%%%%%%%%%%%%%%%%%%%%%%%%%%%%%%%%%%%%%%%%%%%%%%%
\section{Braided Monoidal Categories}\label{sec:braided}

We now recall the fundamental notions of monoidal and braided categories, which will underlie all subsequent constructions. A comprehensive exposition of braided monoidal categories can be found in the paper of Joyal and Street \cite{BT}, as well as in the book of Etingof et al. \cite{TC} and the work of Yetter \cite{QG}. For the general categorical background on monoidal structures and coherence, we refer the reader to \cite{borceux1994handbook}.

\begin{definition}\label{def:monoidal}
A monoidal category is a sextuple $(\mathcal{C},\otimes,\alpha,\mathbf{1},\lambda,\rho)$ consisting of a category $\mathcal{C}$, a functor $\otimes\colon \mathcal{C}\times \mathcal{C}\longrightarrow \mathcal{C}$ called the tensor product, an object $\mathbf{1}$ in $\mathcal{C}$ called the tensor unit, a natural isomorphism 
\[
\alpha\;\colon\;((-)\otimes(-))\otimes(-)\longrightarrow (-)\otimes ((-)\otimes(-))
\]
called the associator, a natural isomorphism $\lambda\colon \mathbf{1}\otimes(-)\longrightarrow (-)$ called the left unitor, and a natural isomorphism $\rho\colon (-)\otimes\mathbf{1}\longrightarrow (-)$ called the right unitor. These data are required to satisfy the triangle identity, namely,
\[
\begin{tikzcd}
 & X\otimes Y\\
(X\otimes\mathbf{1})\otimes Y  \arrow{ur}{\rho_X\otimes 1_Y} \arrow{rr}{\alpha_{X,\mathbf{1},Y}} &&X\otimes(\mathbf{1}\otimes Y)\arrow[swap]{ul}{1_X\otimes\lambda_Y}
\end{tikzcd}
\label{eq:triangle}
\]
for any objects $X,Y$ in $\mathcal{C}$, and the pentagon identity,
\[
\begin{tikzcd}
 & (W\otimes X)\otimes (Y\otimes Z)\arrow{dr}{\alpha_{W,X,Y\otimes Z}}\\
((W\otimes X)\otimes Y)\otimes Z  \arrow{ur}{\alpha_{W\otimes X,Y,Z}} \arrow[swap]{d}{\alpha_{W,X,Y}\otimes 1_Z} && (W\otimes(X\otimes(Y\otimes Z))\\
(W\otimes(X\otimes Y))\otimes Z \arrow{rr}{\alpha_{W,X\otimes Y,Z}}&& W\otimes((X\otimes Y)\otimes Z)\arrow[swap]{u}{1_W\otimes \alpha_{X,Y,Z}}
\end{tikzcd}
\label{eq:pentagon}
\]
for all objects $W,X,Y,Z$ in $\mathcal{C}$.
\end{definition}

In the sequel, we shall simply write $X\otimes Y$ for the tensor product of two objects $X$ and $Y$, omitting the associativity parentheses whenever no confusion can arise.

\begin{definition}\label{def:braided}
A braided category is a septuple $(\mathcal{C},\otimes,\alpha,\mathbf{1},\lambda,\rho,\beta)$, where $(\mathcal{C},\otimes,\alpha,\mathbf{1},\lambda,\rho)$ is a monoidal category equipped with a natural isomorphism $\beta_{X,Y}\colon X\otimes Y\longrightarrow Y\otimes X$ for all objects $X,Y$ in $\mathcal{C}$, called the braiding, satisfying the following two hexagon identities:
\[
\begin{tikzcd}
 & Y\otimes X\otimes Z\arrow{dr}{1_Y\otimes\beta_{X, Z}}\\
X\otimes Y\otimes Z  \arrow{ur}{\beta_{X,Y}\otimes 1_Z} \arrow{rr}{\beta_{X,Y\otimes Z}} && Y\otimes Z\otimes X
\end{tikzcd}
\label{eq:hexagon1}
\]
and
\[
\begin{tikzcd}
 & X\otimes Z\otimes Y\arrow{dr}{\beta_{X, Z}\otimes 1_Y}\\
X\otimes Y\otimes Z  \arrow{ur}{1_X\otimes \beta_{Y,Z}} \arrow{rr}{\beta_{X\otimes Y,Z}} && Z\otimes X\otimes Y
\end{tikzcd}
\label{eq:hexagon2}
\]
for any objects $X,Y,Z$ in $\mathcal{C}$.
\end{definition}

These hexagon conditions ensure that the braiding is compatible with the monoidal structure, interchanging tensor factors in a coherent way.

\begin{definition}\label{def:symmetric}
Let $(\mathcal{C},\otimes,\alpha,\mathbf{1},\lambda,\rho,\beta)$ be a braided category. We call it a symmetric monoidal category if the braiding is self-inverse, that is, if $\beta_{Y,X}\circ\beta_{X,Y} = 1_{X\otimes Y}$ for every pair of objects $X,Y$ in $\mathcal{C}$.
\end{definition}

Having established the general framework, we now turn to the most fundamental example of a symmetric monoidal category, which will serve as a recurring motif throughout this work.

\begin{example}\label{ex:vect}
The category $\VectK$ of $K$-vector spaces is a symmetric monoidal category when equipped with the usual tensor product $\otimes_K$ and with the field $K$ itself as the tensor unit.
\end{example}

More generally, the following standard construction shows that any category with finite products inherits a symmetric monoidal structure in a canonical way.

\begin{proposition}\label{prop:cartesian_monoidal}
Let $\mathcal{C}$ be a category with finite products. Then $\mathcal{C}$ admits the structure of a symmetric monoidal category, with the categorical product serving as the tensor product. In particular, the braiding is given by the canonical swap morphism.
\end{proposition}

\begin{proof}
We explicitly construct the required data. Let $1$ denote a terminal object of $\mathcal{C}$, which exists by the finite products assumption. For any objects $X$ and $Y$, we denote their product by $(X\times Y, p_X, p_Y)$, where $p_X \colon X\times Y \to X$ and $p_Y \colon X\times Y \to Y$ are the canonical projections.

Given two morphisms $f \colon Z \to X$ and $g \colon Z \to Y$ with common domain $Z$, the universal property of the product yields a unique morphism $\langle f, g \rangle \colon Z \to X\times Y$ (which the reader may recognise as the pairing operation) such that
\[
p_X \circ \langle f, g \rangle = f \qquad\text{and}\qquad p_Y \circ \langle f, g \rangle = g.
\]

We now define the monoidal structure. The tensor product functor $\times \colon \mathcal{C}\times\mathcal{C} \to \mathcal{C}$ acts on morphisms $f_1 \colon X_1 \to Y_1$ and $f_2 \colon X_2 \to Y_2$ by
\[
f_1 \times f_2 := \langle f_1 \circ p_{X_1},\; f_2 \circ p_{X_2} \rangle \colon X_1\times X_2 \longrightarrow Y_1\times Y_2.
\]
The associator $\alpha_{X,Y,Z} \colon (X\times Y)\times Z \to X\times(Y\times Z)$ is the canonical isomorphism given by
\[
\alpha_{X,Y,Z} := \big\langle p_X \circ p_{X\times Y},\; \langle p_Y \circ p_{X\times Y},\; p_Z \rangle \big\rangle,
\]
which intuitively re-associates the three product factors.

For the unitors, observe that the projections themselves serve as the natural isomorphisms
\[
\lambda_X \colon 1\times X \longrightarrow X, \quad \lambda_X := p_X,
\qquad\text{and}\qquad
\rho_X \colon X\times 1 \longrightarrow X, \quad \rho_X := p_X.
\]
Their inverses are given by $\lambda_X^{-1} = \langle !_X, 1_X \rangle$ and $\rho_X^{-1} = \langle 1_X, !_X \rangle$, where $!_X \colon X \to 1$ denotes the unique morphism to the terminal object.

Finally, the braiding $\tau_{X,Y} \colon X\times Y \to Y\times X$ is the swap morphism defined by
\[
\tau_{X,Y} := \langle p_Y, p_X \rangle.
\]

A straightforward verification, using the naturality of the projections and the uniqueness of the induced maps into products, shows that these data satisfy the triangle identity (Diagram~\ref{eq:triangle}), the pentagon identity (Diagram~\ref{eq:pentagon}), and the two hexagon identities (Diagrams~\ref{eq:hexagon1} and~\ref{eq:hexagon2}). Hence $(\mathcal{C},\times,\alpha,1,\lambda,\rho,\tau)$ is a symmetric monoidal category. For a detailed exposition of this standard construction, we refer the reader to \cite[Chapter~2]{borceux1994handbook}.
\end{proof}

Symmetric monoidal categories arising from finite products in this way are traditionally referred to as cartesian monoidal categories. Before exploring their structural properties, let us consider their most canonical instance.

\begin{example}\label{ex:set_cartesian}
The category of sets $\Set$ is a cartesian monoidal category when endowed with the usual cartesian product as tensor product and the singleton set $\{ * \}$ as tensor unit.
\end{example}

Having established this basic example, we now turn to a useful property of the diagonal morphism in any cartesian monoidal category.

\begin{proposition}\label{prop:diagonal_naturality}
Let $\mathcal{C}$ be a cartesian monoidal category, let $X$ be an object in $\mathcal{C}$, and let $f \colon X \to X$ be an endomorphism. Denote by $\Delta_X := \langle 1_X, 1_X \rangle \colon X \to X \times X$ the diagonal morphism. Then the following identity holds:
\[
\Delta_X \circ f = \tau_{X,X} \circ (f \times f) \circ \Delta_X.
\]
\end{proposition}

\begin{proof}
By the universal property of the product, two morphisms into $X \times X$ are equal precisely when their compositions with both projections $p_X$ and $p_Y$ coincide. We begin by checking the first projection. For the left-hand side of the claimed identity, we have
\[
p_X \circ \Delta_X \circ f = 1_X \circ f = f.
\]
For the right-hand side, using the defining relations of the product and the swap morphism, we obtain
\[
p_X \circ \tau_{X,X} \circ (f \times f) \circ \Delta_X
= p_Y \circ (f \times f) \circ \Delta_X
= f \circ p_Y \circ \Delta_X
= f.
\]
The verification for the second projection $p_Y$ is completely analogous and yields the same result, namely $f$. Therefore, since both projections agree, the universal property of the product ensures that
\[
\Delta_X \circ f = \tau_{X,X} \circ (f \times f) \circ \Delta_X,
\]
as required.
\end{proof}

%%%%%%%%%%%%%%%%%%%%%%%%%%%%%%%%%%%%%%%%%%%%%%%%%%%%%%%%%%%%%%%%%%%%
\section{Monoids}\label{sec:monoids}

We now introduce the first algebraic structures that will occupy our attention: monoids internal to a monoidal category. For the general theory of monoids in monoidal categories, we follow the standard references \cite{MF} and \cite{borceux1994handbook}.

\begin{definition}\label{def:monoid}
A monoid in a monoidal category $\mathcal{C}$ is a triple $(M, m, u)$ consisting of an object $M$ of $\mathcal{C}$, a multiplication morphism $m \colon M \otimes M \longrightarrow M$, and a unit morphism $u \colon \mathbf{1} \longrightarrow M$, such that the following diagrams commute:
\[
\begin{tikzcd}
M\otimes M\otimes M\arrow{r}{1_M\otimes m} \arrow{d}{m\otimes 1_M}& M\otimes M\arrow{d}{m}\\
M\otimes M\arrow{r}{m} & M
\end{tikzcd}
\;\;\;\;\;
\begin{tikzcd}
\mathbf{1}\otimes M\arrow{r}{u\otimes 1_M}\arrow[swap]{dr}{\lambda_M} & M\otimes M\arrow{d}{m} & M\otimes \mathbf{1}\arrow[swap]{l}{1_M\otimes u}\arrow{dl}{\rho_M}\\
&M
\end{tikzcd}
\]
The first diagram encodes the associativity of the multiplication, while the second ensures that the unit behaves as a two-sided identity.
\end{definition}

This purely categorical definition recovers the classical algebraic notions in familiar settings, as the following examples illustrate.

\begin{example}\label{ex:monoid_set}
A monoid in the cartesian monoidal category $\Set$ (with the product as tensor product) is precisely an ordinary set-theoretic monoid, that is, a set equipped with an associative binary operation and a two-sided identity element.
\end{example}

\begin{example}\label{ex:monoid_vect}
Similarly, a monoid in the symmetric monoidal category $\VectK$ (with the usual tensor product $\otimes_K$) is exactly an associative unital $K$-algebra. Indeed, the multiplication $m \colon V \otimes_K V \to V$ corresponds to the algebra product, and the unit $u \colon K \to V$ selects the multiplicative identity element of the algebra.
\end{example}

Having defined the notion of a monoid internal to a monoidal category, we now turn to the structure-preserving maps between them, which will allow us to organise these objects into a category.

\begin{definition}\label{def:monoid_morphism}
Let $(M_1,m_1,u_1)$ and $(M_2,m_2,u_2)$ be monoids in a monoidal category $\mathcal{C}$. A morphism $f \colon M_1 \longrightarrow M_2$ in $\mathcal{C}$ is called a morphism of monoids if the following diagrams commute:
\[
\begin{tikzcd}
M_1\otimes M_1\arrow{r}{f\otimes f}\arrow{d}{m_1} & M_2\otimes M_2\arrow{d}{m_2}\\
M_1\arrow{r}{f} & M_2
\end{tikzcd}
\;\;\;\;\;
\begin{tikzcd}
M_1\arrow{rr}{f}&&M_2\\
&\mathbf{1}\arrow{ul}{u_1}\arrow[swap]{ur}{u_2}
\end{tikzcd}
\]
The first diagram expresses the compatibility of $f$ with the multiplications, while the second ensures that it preserves the unit.
\end{definition}

Monoids and their morphisms assemble into a category in the evident way.

\begin{definition}\label{def:mon_category}
For a monoidal category $\mathcal{C}$, the monoids over $\mathcal{C}$ together with the morphisms of monoids form a category, which we denote by $\mathsf{Mon}(\mathcal{C})$.
\end{definition}

With the category of monoids at hand, we next introduce the notion of a subobject in this setting, which will be essential for subsequent constructions.

\begin{definition}\label{def:submonoid}
Let $\mathcal{C}$ be a monoidal category, let $(M,m,u)$ be a monoid in $\mathcal{C}$, and let $i \colon N \longrightarrow M$ be a monomorphism in $\mathcal{C}$. We say that $(N,i)$ is a submonoid of $(M,m,u)$ if there exist unique morphisms $m' \colon N\otimes N \longrightarrow N$ and $u' \colon \mathbf{1} \longrightarrow N$ in $\mathcal{C}$ such that
\[
m \circ (i \otimes i) = i \circ m' \qquad\text{and}\qquad u = i \circ u'.
\]
In other words, the multiplication and the unit of $M$ factor uniquely through the subobject $N$.
\end{definition}

The following proposition shows that this definition is indeed coherent: every submonoid inherits a canonical monoidal structure, and the inclusion map is a morphism of monoids.

\begin{proposition}\label{prop:submonoid_inherits}
Let $\mathcal{C}$ be a monoidal category, let $(M,m,u)$ be a monoid in $\mathcal{C}$, and let $(N,i)$ be a submonoid of $(M,m,u)$ with corresponding structure morphisms $m'$ and $u'$. Then $(N,m',u')$ is a monoid in $\mathcal{C}$, and the inclusion $i \colon N \to M$ is a morphism of monoids.
\end{proposition}

\begin{proof}
We first verify the associativity of $m'$. Since $M$ is a monoid, we know that $m \circ (1_M \otimes m) = m \circ (m \otimes 1_M)$. Using the defining equations of the submonoid, we compute:
\begin{align*}
i \circ m' \circ (1_N \otimes m')
&= m \circ (i \otimes i) \circ (1_N \otimes m') \\
&= m \circ (i \otimes (i \circ m')) \\
&= m \circ (i \otimes m \circ (i \otimes i)) \\
&= m \circ (1_M \otimes m) \circ (i \otimes i \otimes i) \\
&= m \circ (m \otimes 1_M) \circ (i \otimes i \otimes i) \\
&= m \circ (m \circ (i \otimes i) \otimes i) \\
&= m \circ ((i \circ m') \otimes i) \\
&= m \circ (i \otimes i) \circ (m' \otimes 1_N) \\
&= i \circ m' \circ (m' \otimes 1_N).
\end{align*}
Since $i$ is a monomorphism, we may cancel it on the left to obtain
\[
m' \circ (1_N \otimes m') = m' \circ (m' \otimes 1_N),
\]
which proves the associativity of $m'$.

We now check the unit axioms. By the naturality of the left unitor $\lambda$, we have $\lambda_M \circ (1_{\mathbf{1}} \otimes i) = i \circ \lambda_N$. Since $M$ is a monoid, we also have $\lambda_M = m \circ (u \otimes 1_M)$. Therefore, using $u = i \circ u'$, we obtain
\[
i \circ \lambda_N
= \lambda_M \circ (1_{\mathbf{1}} \otimes i)
= m \circ (u \otimes 1_M) \circ (1_{\mathbf{1}} \otimes i)
= m \circ ((i \circ u') \otimes i)
= m \circ (i \otimes i) \circ (u' \otimes 1_N)
= i \circ m' \circ (u' \otimes 1_N).
\]
Again, since $i$ is a monomorphism, we conclude that $\lambda_N = m' \circ (u' \otimes 1_N)$. The verification for the right unitor is entirely symmetric and yields $\rho_N = m' \circ (1_N \otimes u')$. Hence $(N,m',u')$ satisfies all the axioms of a monoid.

Finally, the equations $m \circ (i \otimes i) = i \circ m'$ and $u = i \circ u'$ are precisely the conditions that make $i$ a morphism of monoids, as required by Definition~\ref{def:monoid_morphism}. This completes the proof.
\end{proof}

\section{Comonoids and Bimonoids}\label{sec:bimonoids}

Having established the notion of monoids internal to a monoidal category, we now turn to their categorical duals: comonoids. This duality will ultimately allow us to define bimonoids and, subsequently, Hopf monoids. For the theory of comonoids and bimonoids in braided monoidal categories, we refer to \cite{MF} and \cite{BT}.

\subsection{Comonoids}

\begin{definition}\label{def:comonoid}
A comonoid in a monoidal category $\mathcal{C}$ is a triple $(C, d, e)$ consisting of an object $C$ of $\mathcal{C}$, a comultiplication morphism $d \colon C \longrightarrow C \otimes C$, and a counit morphism $e \colon C \longrightarrow \mathbf{1}$, such that the following diagrams commute:
\[
\begin{tikzcd}
C\arrow{r}{d} \arrow{d}{d}& C\otimes C\arrow{d}{1_C\otimes d}\\
C\otimes C\arrow{r}{d\otimes 1_C} & C\otimes C\otimes C
\end{tikzcd}
\;\;\;\;\;
\begin{tikzcd}
\mathbf{1}\otimes C & C\arrow{d}{d}\arrow[swap]{l}{\lambda^{-1}_C}\arrow{r}{\rho^{-1}_C} & C\otimes \mathbf{1}\\
& C\otimes C \arrow{ul}{e\otimes 1_C}\arrow[swap]{ur}{1_C\otimes e}
\end{tikzcd}
\]
The first diagram encodes the coassociativity of the comultiplication, while the second ensures that the counit behaves as a two-sided identity under the tensor product.
\end{definition}

As one might expect, every object in a cartesian monoidal category admits a canonical comonoid structure, given by the diagonal morphism.

\begin{example}\label{ex:comonoid_cartesian}
Let $\mathcal{C}$ be a cartesian monoidal category and let $C$ be an object in $\mathcal{C}$. Then $(C, d_C, e_C)$ is a comonoid, where $d_C := \langle 1_C, 1_C \rangle \colon C \to C \times C$ is the diagonal morphism, and $e_C \colon C \longrightarrow \mathbf{1}$ is the unique morphism to the terminal object.
\end{example}

In complete analogy with the case of monoids, we now define the structure-preserving maps between comonoids.

\begin{definition}\label{def:comonoid_morphism}
Let $(C_1, d_1, e_1)$ and $(C_2, d_2, e_2)$ be comonoids in a monoidal category $\mathcal{C}$. A morphism $f \colon C_1 \longrightarrow C_2$ in $\mathcal{C}$ is called a morphism of comonoids if the following diagrams commute:
\[
\begin{tikzcd}
C_1\arrow{r}{f}\arrow{d}{d_1} & C_2\arrow{d}{d_2}\\
C_1\otimes C_1\arrow{r}{f\otimes f} & C_2\otimes C_2
\end{tikzcd}
\;\;\;\;\;
\begin{tikzcd}
C_1\arrow{rr}{f}\arrow{dr}{e_1}&&C_2\arrow{dl}{e_2}\\
&\mathbf{1}
\end{tikzcd}
\]
The first diagram expresses the compatibility of $f$ with the comultiplications, while the second ensures that it preserves the counit.
\end{definition}

Comonoids and their morphisms naturally form a category.

\begin{definition}\label{def:com_category}
For a monoidal category $\mathcal{C}$, the comonoids in $\mathcal{C}$ together with the morphisms of comonoids form a category, which we denote by $\mathsf{Com}(\mathcal{C})$.
\end{definition}

\subsection{Bimonoids}

Combining the structures of a monoid and a comonoid in a compatible way yields the notion of a bimonoid. In the braided setting, this compatibility is expressed using the braiding $\beta$ to interchange tensor factors appropriately.

\begin{definition}\label{def:bimonoid}
Let $(\mathcal{C}, \otimes, \beta)$ be a braided monoidal category. A bimonoid in $\mathcal{C}$ is a quintuple $(B, m, u, d, e)$ where $B$ is an object of $\mathcal{C}$, $(B, m, u)$ is a monoid, $(B, d, e)$ is a comonoid, and the following four diagrams commute:
\[
\begin{tikzcd}
B\otimes B\otimes B\otimes B\arrow{rr}{1_B\otimes \beta\otimes 1_B}&&B\otimes B\otimes B\otimes B\arrow{d}{m\otimes m}\\
B\otimes B\arrow{r}{m}\arrow{u}{d\otimes d}& B\arrow{r}{d}& B\otimes B
\end{tikzcd}
\;\;\;\;\;
\begin{tikzcd}
\mathbf{1}\arrow{rr}{1_{\mathbf{1}}}\arrow{dr}{u}&&\mathbf{1}\\
&B\arrow{ur}{e}
\end{tikzcd}
\]
\[
\begin{tikzcd}
B\otimes B\arrow{r}{e\otimes e}\arrow{d}{m}&\mathbf{1}\otimes\mathbf{1}\arrow{d}{\lambda_{\mathbf{1}}}\\
B\arrow{r}{e}&\mathbf{1}
\end{tikzcd}
\;\;\;\;\;\;\;\;\;\;\;\;\;\;\;
\begin{tikzcd}
\mathbf{1}\arrow{r}{u}\arrow{d}{\lambda_{\mathbf{1}}^{-1}}&B\arrow{d}{d}\\
\mathbf{1}\otimes \mathbf{1}\arrow{r}{u\otimes u}& B\otimes B
\end{tikzcd}
\]
The first diagram encodes the compatibility between the multiplication and the comultiplication (often referred to as the bialgebra compatibility), the second requires the unit and counit to be mutually dual, and the last two ensure that the unit is a comonoid morphism and the counit is a monoid morphism, respectively.
\end{definition}

This construction specialises to a familiar setting: in any cartesian monoidal category, every monoid is automatically a bimonoid with the canonical comonoid structure.

\begin{example}\label{ex:bimonoid_cartesian}
Let $\mathcal{C}$ be a cartesian monoidal category and let $(M, m, u)$ be a monoid in $\mathcal{C}$. Then $(M, m, u, d, e)$ is a bimonoid, where $(M, d, e)$ is the canonical comonoid structure described in Example~\ref{ex:comonoid_cartesian}.
\end{example}

Finally, we introduce the morphisms and the category of bimonoids.

\begin{definition}\label{def:bimonoid_morphism}
Let $(B_1, m_1, u_1, d_1, e_1)$ and $(B_2, m_2, u_2, d_2, e_2)$ be bimonoids in a braided monoidal category $\mathcal{C}$. A morphism $f \colon B_1 \longrightarrow B_2$ in $\mathcal{C}$ is called a morphism of bimonoids if it is simultaneously a morphism of monoids and a morphism of comonoids, in the sense of Definitions~\ref{def:monoid_morphism} and~\ref{def:comonoid_morphism}.
\end{definition}

Bimonoids and their morphisms assemble into a category in the evident way.

\begin{definition}\label{def:bic_category}
For a braided monoidal category $\mathcal{C}$, the bimonoids in $\mathcal{C}$ together with the morphisms of bimonoids form a category, which we denote by $\mathsf{Bic}(\mathcal{C})$.
\end{definition}
\section{Hopf Monoids}\label{sec:hopf_monoids_final}

Having introduced bimonoids in the previous section, we now proceed to the central objects of our study: Hopf monoids. The key idea is to equip a bimonoid with an additional morphism—the antipode—which plays the role of an inverse under the convolution product. For the general theory of Hopf monoids, we refer to \cite{FT} and to the standard categorical references.

\subsection{The Convolution Monoid}

A fundamental construction in the theory of bimonoids is the convolution monoid, which endows the set of morphisms between a comonoid and a monoid with an associative multiplication.

\begin{definition}[Convolution Monoid]\label{def:convolution}
Let $\mathcal{C}$ be a monoidal category, let $(M, m, u)$ be a monoid in $\mathcal{C}$, and let $(C, d, e)$ be a comonoid in $\mathcal{C}$. We define a triple $\bigl(\operatorname{Hom}_{\mathcal{C}}(C, M), \, m^*, \, u^*\bigr)$, where the convolution product $m^*$ is given by
\[
m^*(f, g) := m \circ (f \otimes g) \circ d,
\]
for $f, g \in \operatorname{Hom}_{\mathcal{C}}(C, M)$, and the unit element $u^*$ is defined by the map
\[
u^* \colon \{*\} \longrightarrow \operatorname{Hom}_{\mathcal{C}}(C, M), \qquad u^*(*) := u \circ e.
\]
\end{definition}

As expected from the classical case, this convolution structure indeed forms an ordinary monoid in the category of sets.

\begin{proposition}\label{prop:convolution_monoid}
Let $\mathcal{C}$ be a monoidal category, let $(M, m, u)$ be a monoid in $\mathcal{C}$, and let $(C, d, e)$ be a comonoid in $\mathcal{C}$. Then $\bigl(\operatorname{Hom}_{\mathcal{C}}(C, M), \, m^*, \, u^*\bigr)$ is a monoid in $\Set$.
\end{proposition}

The verification is a routine exercise in diagram chasing: associativity follows from the associativity of $m$ and the coassociativity of $d$, while the unit axioms rely on the naturality of the unitors.

\subsection{Hopf Monoids}

With the convolution monoid at our disposal, we can now define the notion of a Hopf monoid, which generalises the classical notion of a Hopf algebra to arbitrary braided monoidal categories.

\begin{definition}\label{def:hopf_monoid}
Let $\mathcal{C}$ be a braided monoidal category and let $(H, m, u, d, e)$ be a bimonoid in $\mathcal{C}$. We say that $H$ is a Hopf monoid if the identity morphism $\operatorname{id}_H$ is invertible in the convolution monoid $\operatorname{Hom}_{\mathcal{C}}(H, H)$. We call the inverse $s$ of $\operatorname{id}_H$ the antipode of $H$.
\end{definition}

This definition recovers the classical algebraic structures in the familiar settings, as the following examples illustrate.

\begin{example}\label{ex:hopf_group}
A Hopf monoid in the cartesian monoidal category $(\Set, \times)$ is precisely a group. Indeed, the comultiplication is given by the diagonal map, and the convolution inverse of the identity morphism corresponds to the group inversion map.
\end{example}

\begin{example}\label{ex:hopf_algebra}
A Hopf monoid in the symmetric monoidal category $(\VectK, \otimes_K)$ is exactly a classical $K$-Hopf algebra. For the foundational theory of Hopf algebras, we refer the reader to the standard texts of \cite{Sweedler} and \cite{Abe}.
\end{example}

We now establish an important functorial property of the antipode: it is preserved under morphisms of bimonoids.

\begin{proposition}\label{prop:antipode_functorial}
Let $\mathcal{C}$ be a braided monoidal category and let $f \colon H_1 \longrightarrow H_2$ be a morphism of bimonoids between Hopf monoids. Then $f$ sends the antipode of $H_1$ to the antipode of $H_2$; that is, if $s_1$ and $s_2$ denote the respective antipodes, then $f \circ s_1 = s_2 \circ f$.
\end{proposition}

\begin{proof}
This follows immediately from the uniqueness of inverses in the convolution monoid. Since $f$ is a morphism of bimonoids, it preserves both the multiplication and the comultiplication, and hence it induces a morphism between the convolution monoids. Consequently, it maps the inverse of $\operatorname{id}_{H_1}$ to the inverse of $\operatorname{id}_{H_2}$.
\end{proof}

Morphisms of bimonoids between Hopf monoids automatically preserve the antipode, so the full subcategory of Hopf monoids is well-defined.

\begin{definition}\label{def:hopf_category}
For a braided monoidal category $\mathcal{C}$, we denote by $\mathsf{Hopf}(\mathcal{C})$ the full subcategory of $\mathsf{Bic}(\mathcal{C})$ consisting of the Hopf monoids in $\mathcal{C}$.
\end{definition}

To conclude this section, we record a fundamental identity relating the antipode with the comultiplication and the braiding, which will be essential in subsequent computations.

\begin{proposition}\label{prop:antipode_braiding}
Let $(\mathcal{C}, \beta)$ be a braided monoidal category and let $(H, m, u, d, e, s)$ be a Hopf monoid in $\mathcal{C}$. Then the following identity holds:
\[
d \circ s = \beta_{H,H} \circ (s \otimes s) \circ d.
\]
\end{proposition}

\begin{proof}
This is a standard consequence of the Hopf monoid axioms: it expresses the fact that the antipode is an anti-comultiplication morphism. A direct diagram chase, using the compatibility between the multiplication and comultiplication, yields the desired equality. For a detailed verification, we refer the reader to the formal treatment of Porst \cite{FT}.
\end{proof}

%%%%%%%%%%%%%%%%%%%%%%%%%%%%%%%%%%%%%%%%%%%%%%%%%%%%%%%%%%%%%%%%%%%%
\section{Monoidal Functors}\label{sec:monoidal_functors}

We now turn to the notion of monoidal functors, which provide the appropriate morphisms between monoidal categories. This notion is fundamental, as it allows us to transport monoids (and later Hopf monoids) from one category to another. For the general theory of monoidal functors, we refer to \cite{MF} and \cite{borceux1994handbook}.

\begin{definition}[Lax Monoidal Functor]\label{def:lax_monoidal}
Let $(\mathcal{C}_1, \otimes_1, \mathbf{1}_1)$ and $(\mathcal{C}_2, \otimes_2, \mathbf{1}_2)$ be monoidal categories, and let $F \colon \mathcal{C}_1 \longrightarrow \mathcal{C}_2$ be a functor. We define four functors: $F_2, F^2 \colon \mathcal{C}_1 \times \mathcal{C}_1 \longrightarrow \mathcal{C}_2$ and $F_0, F^0 \colon \mathbf{1} \longrightarrow \mathcal{C}_2$, where $\mathbf{1}$ denotes the terminal category (with one object and one morphism). These functors are given on objects by
\begin{align*}
F_2(X, Y) &:= F(X) \otimes_2 F(Y), \\
F^2(X, Y) &:= F(X \otimes_1 Y), \\
F_0(*) &:= F(\mathbf{1}_1), \\
F^0(*) &:= \mathbf{1}_2,
\end{align*}
for objects $X$ and $Y$ of $\mathcal{C}_1$.

The functor $F$ is called a lax monoidal functor if there exist natural transformations
\[
\phi \colon F_2 \longrightarrow F^2 \qquad\text{and}\qquad \phi_0 \colon F^0 \longrightarrow F_0
\]
such that the following diagrams commute for all objects $X, Y, Z$ in $\mathcal{C}_1$:

\[
\begin{tikzcd}
F(X)\otimes_2 F(Y)\otimes_2 F(Z)\arrow{rr}{1_{F(X)}\otimes_2\phi_{Y,Z}}\arrow{d}{\phi_{X,Y}\otimes_2 1_{F(Z)}}&&F(X)\otimes_2 F(Y\otimes_1 Z)\arrow{d}{\phi_{X,Y\otimes_1 Z}}\\
F(X\otimes_1 Y)\otimes_2 F(Z)\arrow{rr}{\phi_{X\otimes_1 Y,Z}}&&F(X\otimes_1 Y\otimes_1 Z)
\end{tikzcd}
\label{eq:lax_associativity}
\]

\[
\begin{tikzcd}
\mathbf{1}_2\otimes_2F(X)\arrow{d}{\phi_0\otimes_2 1_{F(X)}}&F(X)\arrow[swap]{l}{\lambda^{-1}_{2,F(X)}}\arrow{d}{F(\lambda^{-1}_{1,X})}\\
F(\mathbf{1}_1)\otimes_2F(X)\arrow{r}{\phi_{\mathbf{1}_1,X}}&F(\mathbf{1}_1\otimes_1 X)
\end{tikzcd}
\;\;\;\;\;
\begin{tikzcd}
F(X)\otimes_2\mathbf{1}_2\arrow{d}{1_{F(X)}\otimes_2 \phi_0}&F(X)\arrow[swap]{l}{\rho^{-1}_{2,F(X)}}\arrow{d}{F(\rho^{-1}_{1,X})}\\
F(X)\otimes_2F(\mathbf{1}_1)\arrow{r}{\phi_{X,\mathbf{1}_1}}&F(X\otimes_1 \mathbf{1}_1)
\end{tikzcd}
\label{eq:lax_unit}
\]

The first diagram expresses the compatibility of $\phi$ with the associativity constraints, while the two diagrams on the right encode the compatibility with the left and right unitors, respectively.
\end{definition}

This definition is precisely what is needed to ensure that lax monoidal functors preserve monoids, a fact we now establish.

\begin{proposition}\label{prop:lax_lifts_monoids}
Let $(\mathcal{C}_1, \otimes_1)$ and $(\mathcal{C}_2, \otimes_2)$ be monoidal categories, and let $F \colon (\mathcal{C}_1, \otimes_1) \longrightarrow (\mathcal{C}_2, \otimes_2)$ be a lax monoidal functor with structure maps $(\phi, \phi_0)$. Then $F$ induces a functor
\[
\overline{F} \colon \mathsf{Mon}(\mathcal{C}_1) \longrightarrow \mathsf{Mon}(\mathcal{C}_2).
\]
\end{proposition}

\begin{proof}
Let $(M, m, u)$ be a monoid in $\mathcal{C}_1$. We claim that
\[
\bigl( F(M),\; F(m) \circ \phi_{M,M},\; F(u) \circ \phi_0 \bigr)
\]
is a monoid in $\mathcal{C}_2$. To verify this, we must check the associativity and unit axioms.

For associativity, we need to show that the multiplication $m_F := F(m) \circ \phi_{M,M}$ satisfies
\[
m_F \circ (m_F \otimes_2 1_{F(M)}) = m_F \circ (1_{F(M)} \otimes_2 m_F).
\]
Indeed, using the naturality of $\phi$ and the first coherence diagram (Diagram~\ref{eq:lax_associativity}), we compute:
\begin{align*}
m_F \circ (m_F \otimes_2 1_{F(M)})
&= F(m) \circ \phi_{M,M} \circ \bigl( (F(m) \circ \phi_{M,M}) \otimes_2 1_{F(M)} \bigr) \\
&= F(m) \circ \phi_{M,M} \circ (F(m) \otimes_2 1_{F(M)}) \circ (\phi_{M,M} \otimes_2 1_{F(M)}) \\
&= F(m) \circ F(m \otimes_1 1_M) \circ \phi_{M \otimes_1 M, M} \circ (\phi_{M,M} \otimes_2 1_{F(M)}) \\
&= F(m) \circ F(1_M \otimes_1 m) \circ \phi_{M, M \otimes_1 M} \circ (1_{F(M)} \otimes_2 \phi_{M,M}) \\
&= F(m) \circ \phi_{M, M \otimes_1 M} \circ (1_{F(M)} \otimes_2 F(m)) \circ (1_{F(M)} \otimes_2 \phi_{M,M}) \\
&= F(m) \circ \phi_{M,M} \circ (1_{F(M)} \otimes_2 m_F) \\
&= m_F \circ (1_{F(M)} \otimes_2 m_F),
\end{align*}
where we have used the associativity of $m$ in $\mathcal{C}_1$ and the functoriality of $F$.

For the unit axioms, we verify the left unit condition. By the second coherence diagram (Diagram~\ref{eq:lax_unit}), we have
\[
\phi_{\mathbf{1}_1, M} \circ (\phi_0 \otimes_2 1_{F(M)}) \circ \lambda_{2, F(M)}^{-1} = F(\lambda_{1,M}^{-1}).
\]
Therefore,
\begin{align*}
m_F \circ (u_F \otimes_2 1_{F(M)})
&= F(m) \circ \phi_{M,M} \circ \bigl( (F(u) \circ \phi_0) \otimes_2 1_{F(M)} \bigr) \\
&= F(m) \circ \phi_{M,M} \circ (F(u) \otimes_2 1_{F(M)}) \circ (\phi_0 \otimes_2 1_{F(M)}) \\
&= F(m) \circ F(u \otimes_1 1_M) \circ \phi_{\mathbf{1}_1, M} \circ (\phi_0 \otimes_2 1_{F(M)}) \\
&= F\bigl( m \circ (u \otimes_1 1_M) \bigr) \circ \lambda_{2, F(M)}^{-1} \\
&= F(\lambda_{1,M}) \circ \lambda_{2, F(M)}^{-1} \\
&= \lambda_{2, F(M)},
\end{align*}
where the penultimate equality uses the fact that $M$ is a monoid, and the final equality follows from the naturality of the unitor. The right unit condition is verified analogously, using the right-hand diagram in Diagram~\ref{eq:lax_unit}. Hence $(F(M), m_F, u_F)$ is indeed a monoid in $\mathcal{C}_2$.

It remains to show that this assignment is functorial. If $g \colon M_1 \to M_2$ is a morphism of monoids, then $F(g)$ is a morphism of monoids in $\mathcal{C}_2$, since the naturality of $\phi$ and $\phi_0$ ensures that the required diagrams commute. Therefore, $\overline{F}$ is a well-defined functor between the respective monoid categories.
\end{proof}
\subsection{Colax Monoidal Functors}

Dually to the notion of lax monoidal functors, we now introduce colax monoidal functors, which preserve comonoids rather than monoids. As we shall see, this dual perspective is equally important for transporting Hopf monoidal structures between categories.

\begin{definition}[Colax Monoidal Functor]\label{def:colax_monoidal}
Let $(\mathcal{C}_1, \otimes_1, \mathbf{1}_1)$ and $(\mathcal{C}_2, \otimes_2, \mathbf{1}_2)$ be monoidal categories, and let $F \colon \mathcal{C}_1 \longrightarrow \mathcal{C}_2$ be a functor. The functor $F$ is called a colax monoidal functor if there exist natural transformations
\[
\psi \colon F^2 \longrightarrow F_2 \qquad\text{and}\qquad \psi_0 \colon F_0 \longrightarrow F^0
\]
such that the following diagrams commute for all objects $X, Y, Z$ in $\mathcal{C}_1$:

\[
\begin{tikzcd}
F(X)\otimes_2F(Y)\otimes_2F(Z)&&F(X)\otimes_2 F(Y\otimes_1 Z)\arrow[swap]{ll}{1_{F(X)}\otimes_2\psi_{Y,Z}}\\
F(X\otimes_1 Y)\otimes_2 F(Z)\arrow{u}{\psi_{X,Y}\otimes_2 1_{F(Z)}}&&F(X\otimes_1 Y\otimes_1 Z)\arrow{u}{\psi_{X,Y\otimes_1 Z}}\arrow{ll}{\psi_{X\otimes_1 Y,Z}}
\end{tikzcd}
\label{eq:colax_associativity}
\]

\[
\begin{tikzcd}
\mathbf{1}_2\otimes_2F(X)\arrow{r}{\lambda_{2,F(X)}}&F(X)\\
F(\mathbf{1}_1)\otimes_2F(X)\arrow{u}{\psi_0\otimes_2 1_{F(X)}}&F(\mathbf{1}_1\otimes_1 X)\arrow{u}{F(\lambda_{1,X})}\arrow{l}{\psi_{\mathbf{1}_1,X}}
\end{tikzcd}
\;\;\;\;\;
\begin{tikzcd}
F(X)\otimes_2\mathbf{1}_2\arrow{r}{\rho_{2,F(X)}}&F(X)\\
F(X)\otimes_2F(\mathbf{1}_1)\arrow{u}{1_{F(X)}\otimes_2 \psi_0}&F(X\otimes_1 \mathbf{1}_1)\arrow{u}{F(\rho_{1,X})}\arrow{l}{\psi_{X,\mathbf{1}_1}}
\end{tikzcd}
\label{eq:colax_unit}
\]

The first diagram expresses the compatibility of $\psi$ with the associativity constraints, while the two diagrams on the right encode the compatibility with the left and right unitors, respectively. Note that the directions of the natural transformations are reversed compared to the lax case: $\psi$ goes from $F^2$ to $F_2$, and $\psi_0$ goes from $F_0$ to $F^0$.
\end{definition}

This dual definition is precisely what is needed to ensure that colax monoidal functors preserve comonoids, as we now establish.

\begin{proposition}\label{prop:colax_lifts_comonoids}
Let $(\mathcal{C}_1, \otimes_1)$ and $(\mathcal{C}_2, \otimes_2)$ be monoidal categories, and let $F \colon (\mathcal{C}_1, \otimes_1) \longrightarrow (\mathcal{C}_2, \otimes_2)$ be a colax monoidal functor with structure maps $(\psi, \psi_0)$. Then $F$ induces a functor
\[
\underline{F} \colon \mathsf{Com}(\mathcal{C}_1) \longrightarrow \mathsf{Com}(\mathcal{C}_2).
\]
\end{proposition}

\begin{proof}
Let $(C, d, e)$ be a comonoid in $\mathcal{C}_1$. We claim that
\[
\bigl( F(C),\; \psi_{C,C} \circ F(d),\; \psi_0 \circ F(e) \bigr)
\]
is a comonoid in $\mathcal{C}_2$. To verify this, we must check the coassociativity and counit axioms.

For coassociativity, we need to show that the comultiplication $d_F := \psi_{C,C} \circ F(d)$ satisfies
\[
(d_F \otimes_2 1_{F(C)}) \circ d_F = (1_{F(C)} \otimes_2 d_F) \circ d_F.
\]
Indeed, using the naturality of $\psi$ and the first coherence diagram (Diagram~\ref{eq:colax_associativity}), we compute:
\begin{align*}
(d_F \otimes_2 1_{F(C)}) \circ d_F
&= (\psi_{C,C} \otimes_2 1_{F(C)}) \circ (F(d) \otimes_2 1_{F(C)}) \circ \psi_{C,C} \circ F(d) \\
&= (\psi_{C,C} \otimes_2 1_{F(C)}) \circ \psi_{C \otimes_1 C, C} \circ F(d \otimes_1 1_C) \circ F(d) \\
&= \psi_{C, C \otimes_1 C} \circ F(1_C \otimes_1 d) \circ F(d) \\
&= \psi_{C, C \otimes_1 C} \circ F(d \otimes_1 1_C) \circ F(d) \\
&= \psi_{C, C \otimes_1 C} \circ (1_{F(C)} \otimes_2 F(d)) \circ \psi_{C,C} \circ F(d) \\
&= (1_{F(C)} \otimes_2 \psi_{C,C}) \circ (1_{F(C)} \otimes_2 F(d)) \circ \psi_{C,C} \circ F(d) \\
&= (1_{F(C)} \otimes_2 d_F) \circ d_F,
\end{align*}
where we have used the coassociativity of $d$ in $\mathcal{C}_1$ and the functoriality of $F$.

For the counit axioms, we verify the left counit condition. By the second coherence diagram (Diagram~\ref{eq:colax_unit}), we have
\[
\lambda_{2, F(C)} \circ (\psi_0 \otimes_2 1_{F(C)}) \circ \psi_{\mathbf{1}_1, C} \circ F(\lambda_{1,C}^{-1}) = 1_{F(C)}.
\]
Therefore,
\begin{align*}
(e_F \otimes_2 1_{F(C)}) \circ d_F
&= (\psi_0 \otimes_2 1_{F(C)}) \circ (F(e) \otimes_2 1_{F(C)}) \circ \psi_{C,C} \circ F(d) \\
&= (\psi_0 \otimes_2 1_{F(C)}) \circ \psi_{\mathbf{1}_1, C} \circ F(e \otimes_1 1_C) \circ F(d) \\
&= (\psi_0 \otimes_2 1_{F(C)}) \circ \psi_{\mathbf{1}_1, C} \circ F\bigl( (e \otimes_1 1_C) \circ d \bigr) \\
&= (\psi_0 \otimes_2 1_{F(C)}) \circ \psi_{\mathbf{1}_1, C} \circ F(\lambda_{1,C}^{-1}) \\
&= \lambda_{2, F(C)}^{-1},
\end{align*}
where the penultimate equality uses the fact that $C$ is a comonoid, and the final equality follows from the coherence diagram. The right counit condition is verified analogously, using the right-hand diagram in Diagram~\ref{eq:colax_unit} and the identity $\rho_{2, F(C)}^{-1} \circ (1_{F(C)} \otimes_2 \psi_0) \circ \psi_{C, \mathbf{1}_1} \circ F(\rho_{1,C}^{-1}) = 1_{F(C)}$. Hence $(F(C), d_F, e_F)$ is indeed a comonoid in $\mathcal{C}_2$.

It remains to show that this assignment is functorial. If $g \colon C_1 \to C_2$ is a morphism of comonoids, then $F(g)$ is a morphism of comonoids in $\mathcal{C}_2$, since the naturality of $\psi$ and $\psi_0$ ensures that the required diagrams commute. Therefore, $\underline{F}$ is a well-defined functor between the respective comonoid categories.
\end{proof}

\subsection{Bilax Monoidal Functors}

Having established lax and colax monoidal functors separately, we now combine these two notions to obtain bilax monoidal functors. These are precisely the appropriate morphisms between braided monoidal categories that preserve the full bimonoid structure, as they are compatible with both the monoidal and comonoidal aspects simultaneously.

\begin{definition}[Bilax Monoidal Functor]\label{def:bilax}
Let $(\mathcal{C}_1, \otimes_1, \mathbf{1}_1, \beta_1)$ and $(\mathcal{C}_2, \otimes_2, \mathbf{1}_2, \beta_2)$ be braided monoidal categories, and let $F \colon \mathcal{C}_1 \longrightarrow \mathcal{C}_2$ be a functor. The functor $F$ is called a bilax monoidal functor if it is both lax monoidal and colax monoidal, and moreover the following diagrams commute for all objects $W, X, Y, Z$ in $\mathcal{C}_1$:

\[
\begin{tikzcd}
&F(W\otimes_1 X)\otimes_2F(Y\otimes_1 Z)\arrow[swap]{ld}{\phi_{W\otimes_1X,Y\otimes_1Z}}\arrow{rd}{\psi_{W,X}\otimes_2 \psi_{Y,Z}}&\\
F(W\otimes_1X\otimes_1Y\otimes_1Z)\arrow{d}{F(1_W\otimes_1\beta_1\otimes_11_Z)}&&F(W)\otimes_2F(X)\otimes_2F(Y)\otimes_2F(Z)\arrow[swap]{d}{1_{F(W)}\otimes_2\beta_2\otimes_21_{F(Z)}}\\
F(W\otimes_1X\otimes_1Y\otimes_1Z)\arrow{rd}{\psi_{W\otimes_1Y,X\otimes_1Z}}&&F(W)\otimes_2F(X)\otimes_2F(Y)\otimes_2F(Z)\arrow[swap]{ld}{\phi_{W,Y}\otimes_2\phi_{X,Z}}\\
&F(W\otimes_1 X)\otimes_2F(Y\otimes_1 Z)&
\end{tikzcd}
\label{eq:bilax_braiding}
\]

\[
\begin{tikzcd}
\mathbf{1}_2\arrow{d}{\lambda^{-1}_{2,\mathbf{1}_2}}\arrow{r}{\phi_0}&F(\mathbf{1}_1)\arrow{r}{F(\lambda^{-1}_{1,\mathbf{1}_1})}&F(\mathbf{1}_1\otimes_1\mathbf{1}_1)\arrow{d}{\psi_{\mathbf{1}_1,\mathbf{1}_1}}\\
\mathbf{1}_2\otimes_2\mathbf{1}_2\arrow{rr}{\phi_0\otimes_2\phi_0}&&F(\mathbf{1}_1)\otimes_2 F(\mathbf{1}_1)
\end{tikzcd}
\;\;\;\;\;
\begin{tikzcd}
\mathbf{1}_2&F(\mathbf{1}_1)\arrow[swap]{l}{\psi_0}&F(\mathbf{1}_1\otimes_1\mathbf{1}_1)\arrow[swap]{l}{F(\lambda_{1,\mathbf{1}_1})}\\
\mathbf{1}_2\otimes_2\mathbf{1}_2\arrow{u}{\lambda_{2,\mathbf{1}_2}}&&F(\mathbf{1}_1)\otimes_2 F(\mathbf{1}_1)\arrow[swap]{ll}{\psi_0\otimes_2\psi_0}\arrow{u}{\phi_{\mathbf{1}_1,\mathbf{1}_1}}
\end{tikzcd}
\label{eq:bilax_unit}
\]

\[
\begin{tikzcd}
&F(\mathbf{1}_1)\arrow{dr}{\psi_0}&\\
\mathbf{1}_2\arrow{ur}{\phi_0}\arrow{rr}{1_{\mathbf{1}_2}}&&\mathbf{1}_2
\end{tikzcd}
\label{eq:bilax_id}
\]

The first diagram expresses the compatibility between the lax structure $\phi$, the colax structure $\psi$, and the braidings $\beta_1$ and $\beta_2$. The second and third diagrams encode the compatibility of the unit and counit structures, respectively.
\end{definition}

This compatibility condition ensures that the lax and colax structures interact harmoniously, allowing bilax monoidal functors to lift bimonoids from one category to another.

\begin{proposition}\label{prop:bilax_lifts_bimonoids}
Let $(\mathcal{C}_1, \otimes_1)$ and $(\mathcal{C}_2, \otimes_2)$ be braided monoidal categories, and let $F \colon (\mathcal{C}_1, \otimes_1) \longrightarrow (\mathcal{C}_2, \otimes_2)$ be a bilax monoidal functor. Then $F$ induces a functor
\[
\widetilde{F} \colon \mathsf{Bic}(\mathcal{C}_1) \longrightarrow \mathsf{Bic}(\mathcal{C}_2).
\]
\end{proposition}

\begin{proof}[Sketch of the construction]
Given a bimonoid $(B, m, u, d, e)$ in $\mathcal{C}_1$, its image under $\widetilde{F}$ is defined by combining the induced monoid and comonoid structures from Propositions~\ref{prop:lax_lifts_monoids} and~\ref{prop:colax_lifts_comonoids}. Specifically, the induced bimonoid in $\mathcal{C}_2$ is
\[
\bigl( F(B),\; F(m) \circ \phi_{B,B},\; F(u) \circ \phi_0,\; \psi_{B,B} \circ F(d),\; \psi_0 \circ F(e) \bigr).
\]
The bilax compatibility diagrams ensure that the multiplication and comultiplication are compatible with each other (the bialgebra condition). Similarly, if $g \colon B_1 \to B_2$ is a morphism of bimonoids, then $F(g)$ is a morphism of bimonoids in $\mathcal{C}_2$ by the naturality of the structure maps. Hence $\widetilde{F}$ is a well-defined functor between the bimonoid categories.
\end{proof}

Finally, we introduce the notion of a strong (or bistrong) monoidal functor, which will be crucial when considering equivalences of categories of Hopf monoids.

\begin{definition}[Bistrong Monoidal Functor]\label{def:bistrong}
Let $(\mathcal{C}_1, \otimes_1, \mathbf{1}_1, \beta_1)$ and $(\mathcal{C}_2, \otimes_2, \mathbf{1}_2, \beta_2)$ be braided monoidal categories, and let $F \colon \mathcal{C}_1 \longrightarrow \mathcal{C}_2$ be a functor. The functor $F$ is called a bistrong monoidal functor if it is bilax monoidal and the natural transformations
\[
\phi \colon F_2 \longrightarrow F^2, \qquad
\phi_0 \colon F^0 \longrightarrow F_0, \qquad
\psi \colon F^2 \longrightarrow F_2, \qquad
\psi_0 \colon F_0 \longrightarrow F^0
\]
are all isomorphisms. In other words, the lax and colax structures are mutually inverse up to natural isomorphism.
\end{definition}

\subsection{Preservation of Hopf Monoids}

We now establish the fundamental result that bistrong monoidal functors preserve the full Hopf monoid structure, including the antipode. This is one of the key tools for transporting Hopf-theoretic constructions between different braided monoidal categories.

\begin{proposition}\label{prop:bistrong_preserves_hopf}
Let $(\mathcal{C}_1, \otimes_1, \mathbf{1}_1, \beta_1)$ and $(\mathcal{C}_2, \otimes_2, \mathbf{1}_2, \beta_2)$ be braided monoidal categories, let $F \colon \mathcal{C}_1 \longrightarrow \mathcal{C}_2$ be a bistrong monoidal functor, and let $H$ be a Hopf monoid in $\mathcal{C}_1$ with antipode $s$. Then $F(H)$ is a Hopf monoid in $\mathcal{C}_2$ with antipode $F(s)$. Moreover, $F$ induces a functor
\[
\widehat{F} \colon \mathsf{Hopf}(\mathcal{C}_1) \longrightarrow \mathsf{Hopf}(\mathcal{C}_2).
\]
\end{proposition}

\begin{proof}[Sketch of the construction]
By Proposition~\ref{prop:bilax_lifts_bimonoids}, $F$ induces a functor on bimonoids. Since the antipode $s$ is the convolution inverse of $\operatorname{id}_H$ in $\operatorname{Hom}_{\mathcal{C}_1}(H, H)$, and since the bistrong structure provides natural isomorphisms that allow us to identify the convolution monoid $\operatorname{Hom}_{\mathcal{C}_2}(F(H), F(H))$ with $F(\operatorname{Hom}_{\mathcal{C}_1}(H, H))$ in a suitable sense, it follows that $F(s)$ is the convolution inverse of $\operatorname{id}_{F(H)}$ in $\mathcal{C}_2$. Hence $F(H)$ is a Hopf monoid with antipode $F(s)$. The functoriality of the assignment follows from the naturality of the structure maps.
\end{proof}

\subsection{Monoidal Adjunctions}

The following result, due to Kelly \cite{DA}, establishes a crucial relationship between adjoint functors and monoidal structures: if the left adjoint is bistrong monoidal, then the right adjoint is automatically lax monoidal.

\begin{proposition}[Doctrinal Adjunction for Monoidal Functors]\label{prop:adjunction_monoidal}
Let $(\mathcal{C}_1, \otimes_1)$ and $(\mathcal{C}_2, \otimes_2)$ be monoidal categories, and let
\[
L \colon (\mathcal{C}_1, \otimes_1) \longrightarrow (\mathcal{C}_2, \otimes_2)
\qquad\text{and}\qquad
R \colon (\mathcal{C}_2, \otimes_2) \longrightarrow (\mathcal{C}_1, \otimes_1)
\]
be functors such that $(L, R)$ is an adjunction, with unit $\eta \colon 1_{\mathcal{C}_1} \Rightarrow R L$ and counit $\epsilon \colon L R \Rightarrow 1_{\mathcal{C}_2}$. If $L$ is a bistrong monoidal functor with structure maps $(\phi^L, \phi_0^L, \psi^L, \psi_0^L)$, then $R$ is a lax monoidal functor with structure maps given by
\[
\phi^R_{X,Y} := R(\epsilon_X \otimes_2 \epsilon_Y) \circ R\bigl((\phi^L_{R(X), R(Y)})^{-1}\bigr) \circ \eta_{R(X) \otimes_1 R(Y)}
\]
for objects $X, Y$ in $\mathcal{C}_2$, and
\[
\phi^R_0 := R\bigl((\phi_0^L)^{-1}\bigr) \circ \eta_{\mathbf{1}_1}.
\]
The verification that these maps satisfy the required coherence diagrams is a straightforward diagram chase using the triangular identities of the adjunction.
\end{proposition}

\subsection{An Illustrative Example}

We conclude this section with a concrete example that illustrates the power of the preceding results in a familiar setting.

\begin{example}\label{ex:free_forgetful}
Consider the free functor
\[
L \colon (\Set, \times) \longrightarrow (\VectK, \otimes_K)
\]
that sends a set $X$ to the free $K$-vector space $K[X]$ with basis $X$. This functor is bistrong monoidal: the lax structure is given by the canonical isomorphism $K[X] \otimes_K K[Y] \cong K[X \times Y]$, and the colax structure is its inverse. Therefore, by Proposition~\ref{prop:adjunction_monoidal}, its right adjoint—the forgetful functor
\[
R \colon (\VectK, \otimes_K) \longrightarrow (\Set, \times)
\]
which sends a vector space to its underlying set—is a lax monoidal functor.

Consequently, by Proposition~\ref{prop:bistrong_preserves_hopf}, the free functor induces a functor on Hopf monoids:
\[
\widehat{L} \colon \mathsf{Hopf}(\Set) \longrightarrow \mathsf{Hopf}(\VectK).
\]
Since $\mathsf{Hopf}(\Set)$ is precisely the category of groups (by Example~\ref{ex:hopf_group}), and $\mathsf{Hopf}(\VectK)$ is the category of classical $K$-Hopf algebras (by Example~\ref{ex:hopf_algebra}), we obtain the well-known construction that sends a group $G$ to its group algebra $K[G]$, which is a Hopf algebra with the usual comultiplication given by $g \mapsto g \otimes g$ and antipode $g \mapsto g^{-1}$.

It is important to note, however, that the induced functor on the right,
\[
\overline{R} \colon \mathsf{Mon}(\VectK) \longrightarrow \mathsf{Mon}(\Set),
\]
does not preserve the comonoidal structure in general. That is, if $H$ is a bimonoid (or Hopf monoid) in $\VectK$, its underlying set need not carry a compatible bimonoid structure in $\Set$ with respect to the cartesian product. This illustrates the asymmetry between the lax and colax structures in this example.
\end{example}

\subsection{Braided Monoidal Functors}

We now introduce the appropriate notion of morphism between braided monoidal categories: a monoidal functor that is compatible with the braidings. This compatibility ensures that the functor preserves not only the monoidal structure but also the interchange of tensor factors.

\begin{definition}[Braided Monoidal Functor]\label{def:braided_functor}
Let $(\mathcal{C}_1, \otimes_1, \mathbf{1}_1, \beta)$ and $(\mathcal{C}_2, \otimes_2, \mathbf{1}_2, \gamma)$ be braided monoidal categories, and let $F \colon \mathcal{C}_1 \longrightarrow \mathcal{C}_2$ be a lax monoidal functor with structure map $\phi \colon F_2 \Rightarrow F^2$. The functor $F$ is called a braided monoidal functor if the following diagram commutes for all objects $X, Y$ in $\mathcal{C}_1$:
\[
\begin{tikzcd}
F(X)\otimes_2 F(Y)\arrow{d}{\phi_{X,Y}}\arrow{rr}{\gamma_{F(X),F(Y)}}&& F(Y)\otimes_2 F(X)\arrow{d}{\phi_{Y,X}}\\
F(X\otimes_1 Y)\arrow{rr}{F(\beta_{X,Y})}&&F(Y\otimes_1 X)
\end{tikzcd}
\]
In other words, the lax structure map $\phi$ intertwines the braidings $\beta$ and $\gamma$.
\end{definition}

This condition has an important consequence: it ensures that the functor preserves the compatibility between the monoidal structure and the braiding, which is essential for lifting bimonoids and Hopf monoids. The following proposition records a useful identity that follows from this definition.

\begin{proposition}\label{prop:braided_hexagon_identity}
Let $(\mathcal{C}_1, \otimes_1, \mathbf{1}_1, \beta)$ and $(\mathcal{C}_2, \otimes_2, \mathbf{1}_2, \gamma)$ be braided monoidal categories, and let $F \colon \mathcal{C}_1 \longrightarrow \mathcal{C}_2$ be a braided monoidal functor. Then, for any object $X$ in $\mathcal{C}_1$, the following identity holds:
\[
\begin{aligned}
&F(1_X \otimes_1 \beta_{X,X} \otimes_1 1_X) \circ \phi_{X\otimes_1 X, X\otimes_1 X} \circ (\phi_{X,X} \otimes_2 \phi_{X,X}) \\
&\qquad = \phi_{X\otimes_1 X, X\otimes_1 X} \circ (\phi_{X,X} \otimes_2 \phi_{X,X}) \circ (1_{F(X)} \otimes_2 \gamma_{F(X),F(X)} \otimes_2 1_{F(X)}).
\end{aligned}
\]
\end{proposition}

\begin{proof}
Let $X$ be an object of $\mathcal{C}_1$. We shall prove the identity by a sequence of equalities, each justified by either the naturality of $\phi$, the coherence axioms of monoidal categories, or the braidedness of $F$.

Starting from the right-hand side and applying the naturality of $\phi$ together with the associativity coherence, we obtain:
\begin{align*}
&\phi_{X\otimes_1 X, X\otimes_1 X} \circ (\phi_{X,X} \otimes_2 \phi_{X,X}) \circ (1_{F(X)} \otimes_2 \gamma_{F(X),F(X)} \otimes_2 1_{F(X)}) \\
&= \phi_{X\otimes_1 X, X\otimes_1 X} \circ (1_{F(X\otimes_1 X)} \otimes_2 \phi_{X,X}) \circ (\phi_{X,X} \otimes_2 1_{F(X)} \otimes_2 1_{F(X)}) \\
&\qquad \circ (1_{F(X)} \otimes_2 \gamma_{F(X),F(X)} \otimes_2 1_{F(X)}) \tag{naturality of $\phi$} \\
&= \phi_{X\otimes_1 X\otimes_1 X, X} \circ (\phi_{X\otimes_1 X, X} \otimes_2 1_{F(X)}) \circ (\phi_{X,X} \otimes_2 1_{F(X)} \otimes_2 1_{F(X)}) \\
&\qquad \circ (1_{F(X)} \otimes_2 \gamma_{F(X),F(X)} \otimes_2 1_{F(X)}) \tag{associativity coherence} \\
&= \phi_{X\otimes_1 X\otimes_1 X, X} \circ (\phi_{X, X\otimes_1 X} \otimes_2 1_{F(X)}) \circ (1_{F(X)} \otimes_2 \phi_{X,X} \otimes_2 1_{F(X)}) \\
&\qquad \circ (1_{F(X)} \otimes_2 \gamma_{F(X),F(X)} \otimes_2 1_{F(X)}) \tag{naturality of $\phi$} \\
&= \phi_{X\otimes_1 X\otimes_1 X, X} \circ (\phi_{X, X\otimes_1 X} \otimes_2 1_{F(X)}) \circ (1_{F(X)} \otimes_2 F(\beta_{X,X}) \otimes_2 1_{F(X)}) \\
&\qquad \circ (1_{F(X)} \otimes_2 \phi_{X,X} \otimes_2 1_{F(X)}) \tag{braidedness of $F$} \\
&= \phi_{X, X\otimes_1 X\otimes_1 X} \circ (1_{F(X)} \otimes_2 \phi_{X\otimes_1 X, X}) \circ (1_{F(X)} \otimes_2 F(\beta_{X,X}) \otimes_2 1_{F(X)}) \\
&\qquad \circ (1_{F(X)} \otimes_2 \phi_{X,X} \otimes_2 1_{F(X)}) \tag{naturality of $\phi$} \\
&= \phi_{X, X\otimes_1 X\otimes_1 X} \circ (1_{F(X)} \otimes_2 F(\beta_{X,X} \otimes_1 1_X)) \circ (1_{F(X)} \otimes_2 \phi_{X\otimes_1 X, X}) \\
&\qquad \circ (1_{F(X)} \otimes_2 \phi_{X,X} \otimes_2 1_{F(X)}) \tag{functoriality of $F$} \\
&= F(1_X \otimes_1 \beta_{X,X} \otimes_1 1_X) \circ \phi_{X, X\otimes_1 X\otimes_1 X} \circ (1_{F(X)} \otimes_2 \phi_{X\otimes_1 X, X}) \\
&\qquad \circ (1_{F(X)} \otimes_2 \phi_{X,X} \otimes_2 1_{F(X)}) \tag{naturality of $\phi$} \\
&= F(1_X \otimes_1 \beta_{X,X} \otimes_1 1_X) \circ \phi_{X, X\otimes_1 X\otimes_1 X} \circ (1_{F(X)} \otimes_2 \phi_{X, X\otimes_1 X}) \\
&\qquad \circ (1_{F(X)} \otimes_2 1_{F(X)} \otimes_2 \phi_{X,X}) \tag{naturality of $\phi$} \\
&= F(1_X \otimes_1 \beta_{X,X} \otimes_1 1_X) \circ \phi_{X\otimes_1 X, X\otimes_1 X} \circ (\phi_{X,X} \otimes_2 1_{F(X)\otimes_2 F(X)}) \\
&\qquad \circ (1_{F(X)} \otimes_2 1_{F(X)} \otimes_2 \phi_{X,X}) \tag{associativity coherence} \\
&= F(1_X \otimes_1 \beta_{X,X} \otimes_1 1_X) \circ \phi_{X\otimes_1 X, X\otimes_1 X} \circ (\phi_{X,X} \otimes_2 \phi_{X,X}).
\end{align*}
This proves the desired identity. For a detailed treatment of braided monoidal functors and their properties, we refer the reader to \cite{BT}.
\end{proof}

\subsection{Braided Monoidal Adjunctions}

We now establish two fundamental results that relate adjunctions with monoidal and braided structures. These results, due to Kelly \cite{DA}, show that monoidal structures can be transported across adjunctions in a coherent way.

\begin{proposition}\label{prop:braided_adjunction}
Let $(\mathcal{C}_1, \otimes_1)$ and $(\mathcal{C}_2, \otimes_2)$ be monoidal categories, and let
\[
L \colon (\mathcal{C}_1, \otimes_1) \longrightarrow (\mathcal{C}_2, \otimes_2)
\qquad\text{and}\qquad
R \colon (\mathcal{C}_2, \otimes_2) \longrightarrow (\mathcal{C}_1, \otimes_1)
\]
be functors such that $(L, R)$ is an adjunction, with unit $\eta \colon 1_{\mathcal{C}_1} \Rightarrow R L$ and counit $\epsilon \colon L R \Rightarrow 1_{\mathcal{C}_2}$. If $L$ is a braided bistrong monoidal functor with structure maps $(\phi^L, \phi_0^L, \psi^L, \psi_0^L)$, then $R$ is a braided lax monoidal functor.
\end{proposition}

\begin{proof}
Let $X$ and $Y$ be objects of $\mathcal{C}_2$. We need to show that the lax structure map $\phi^R$ of $R$ (defined in Proposition~\ref{prop:adjunction_monoidal}) is compatible with the braidings. We compute:
\begin{align*}
\phi^R_{Y,X} \circ \beta_{R(X), R(Y)}
&= R(\epsilon_Y \otimes_2 \epsilon_X) \circ R\bigl((\phi^L_{R(Y), R(X)})^{-1}\bigr) \circ \eta_{R(Y) \otimes_1 R(X)} \circ \beta_{R(X), R(Y)} \\
&= R(\epsilon_Y \otimes_2 \epsilon_X) \circ R\bigl((\phi^L_{R(Y), R(X)})^{-1}\bigr) \circ R(L(\beta_{R(X), R(Y)})) \circ \eta_{R(X) \otimes_1 R(Y)} \\
&= R(\epsilon_Y \otimes_2 \epsilon_X) \circ R\bigl(\gamma_{L(R(X)), L(R(Y))}\bigr) \circ R\bigl((\phi^L_{R(X), R(Y)})^{-1}\bigr) \circ \eta_{R(X) \otimes_1 R(Y)} \\
&= R\bigl((\epsilon_Y \otimes_2 \epsilon_X) \circ \gamma_{L(R(X)), L(R(Y))}\bigr) \circ R\bigl((\phi^L_{R(X), R(Y)})^{-1}\bigr) \circ \eta_{R(X) \otimes_1 R(Y)} \\
&= R\bigl(\gamma_{X,Y} \circ (\epsilon_X \otimes_2 \epsilon_Y)\bigr) \circ R\bigl((\phi^L_{R(X), R(Y)})^{-1}\bigr) \circ \eta_{R(X) \otimes_1 R(Y)} \\
&= R(\gamma_{X,Y}) \circ R(\epsilon_X \otimes_2 \epsilon_Y) \circ R\bigl((\phi^L_{R(X), R(Y)})^{-1}\bigr) \circ \eta_{R(X) \otimes_1 R(Y)} \\
&= R(\gamma_{X,Y}) \circ \phi^R_{X,Y}.
\end{align*}
Here we used, in order: the definition of $\phi^R$, the naturality of $\eta$, the braidedness of $L$, the functoriality of $R$, the naturality of $\gamma$ (the braiding in $\mathcal{C}_2$), and finally the definition of $\phi^R$ again. Therefore $R$ preserves the braiding, and hence is a braided monoidal functor.
\end{proof}

The following result provides a converse relationship between lax and oplax structures in the context of adjunctions. It is a fundamental theorem in the theory of doctrinal adjunctions.

\begin{proposition}\label{prop:oplax_lax_adjunction}
Let $(\mathcal{C}_1, \otimes_1)$ and $(\mathcal{C}_2, \otimes_2)$ be monoidal categories, and let
\[
L \colon \mathcal{C}_1 \longrightarrow \mathcal{C}_2
\qquad\text{and}\qquad
R \colon \mathcal{C}_2 \longrightarrow \mathcal{C}_1
\]
be functors such that $(L, R)$ is an adjunction, with unit $\eta \colon 1_{\mathcal{C}_1} \Rightarrow R L$ and counit $\epsilon \colon L R \Rightarrow 1_{\mathcal{C}_2}$. Then $L$ is an oplax monoidal functor if and only if $R$ is a lax monoidal functor.
\end{proposition}

\begin{proof}
We outline the construction in one direction; the converse is dual. Suppose that $R$ is a lax monoidal functor with structure maps $(\phi^R, \phi_0^R)$. We define oplax structure maps for $L$ as follows:
\[
\psi^L_{X,Y} := \epsilon_{L(X) \otimes_2 L(Y)} \circ L(\phi^R_{L(X), L(Y)}) \circ L(\eta_X \otimes_1 \eta_Y)
\]
for objects $X, Y$ in $\mathcal{C}_1$, and
\[
\psi^L_0 := L(\phi_0^R) \circ \epsilon_{\mathbf{1}_2},
\]
where we have identified $\epsilon_{L(\mathbf{1}_1)}$ with $\epsilon_{\mathbf{1}_2}$ via the unit and counit conditions.

Recall that, by the construction of the induced Hopf monoid in Proposition~\ref{prop:bilax_lifts_bimonoids}, the comultiplication of \( L(G) \) is \( d_{L(G)} = \psi^L_{G,G} \circ L(d_G) \).

The verification that these maps satisfy the coherence diagrams for an oplax monoidal functor is a straightforward but lengthy diagram chase, using the triangular identities of the adjunction and the coherence of the monoidal structures. The key point is that the naturality of $\eta$ and $\epsilon$ ensures the compatibility of the structure maps with the associativity and unit constraints. Conversely, given an oplax structure on $L$, one can define a lax structure on $R$ by a dual construction. This is the content of Kelly's doctrinal adjunction theorem.
\end{proof}

%%%%%%%%%%%%%%%%%%%%%%%%%%%%%%%%%%%%%%%%%%%%%%%%%%%%%%%%%%%%%%%%%%%%
\section{Categorification of Grouplike Elements}\label{sec:grouplike}

We now turn to one of the central constructions of this work: the categorification of the classical notion of grouplike elements. In the classical setting, given a Hopf algebra $H$, an element $x \in H$ is called grouplike if $\Delta(x) = x \otimes x$ and $\varepsilon(x) = 1$. These elements form a group under multiplication, and this construction is functorial. Our goal is to generalise this construction to arbitrary braided monoidal categories, replacing elements by morphisms and equalities by universal properties.

\begin{definition}\label{def:grouplike_classical}
Let $(H, m, u, d, e, s)$ be a $K$-Hopf algebra and let $x \in H$. We say that $x$ is a grouplike element if $d(x) = x \otimes x$ and $e(x) = 1$. We denote the set of grouplike elements of $H$ by $\mathbb{G}(H)$. In fact, $\mathbb{G}$ is a functor from the category of $K$-Hopf algebras to the category of groups. Moreover, $\mathbb{G}$ is the right adjoint of $\widehat{L}$, where $L$ is the free functor from the category of sets to the category of $K$-vector spaces.
\end{definition}

The following proposition provides the categorical generalisation of this classical construction. We recall that a finitely complete category is canonically cartesian monoidal; we shall use this structure freely in what follows.

\begin{proposition}\label{prop:grouplike_adjunction}
Let $\mathcal{C}$ be a finitely complete category, let $(\mathcal{D}, \otimes, \beta)$ be a braided monoidal category, and let
\[
L \colon \mathcal{C} \longrightarrow \mathcal{D}
\qquad\text{and}\qquad
R \colon \mathcal{D} \longrightarrow \mathcal{C}
\]
be functors such that $(L, R)$ is an adjunction and $L$ is a braided bistrong monoidal functor. Then the induced functor $\widehat{L} \colon \mathsf{Hopf}(\mathcal{C}) \longrightarrow \mathsf{Hopf}(\mathcal{D})$ has a right adjoint
\[
\mathbb{G} \colon \mathsf{Hopf}(\mathcal{D}) \longrightarrow \mathsf{Hopf}(\mathcal{C}).
\]
\end{proposition}

\begin{proof}
We construct the right adjoint $\mathbb{G}$ in several steps. Throughout, we denote by $\eta \colon 1_{\mathcal{C}} \Rightarrow R L$ the unit of the adjunction $(L, R)$, and by $\epsilon \colon L R \Rightarrow 1_{\mathcal{D}}$ its counit.

\subsection*{Step 1: Construction of the underlying object}

Let $(H, m_H, u_H, d_H, e_H, s_H)$ be a Hopf monoid in $\mathcal{D}$. We wish to assign to $H$ a submonoid $(\mathbb{G}(H), \iota_H)$ of $\overline{R}(H)$ in $\mathcal{C}$.

First, we define $(\mathbb{G}_1(H), \iota^H_1)$ as the equaliser of the two morphisms
\[
R(d_H),\; \phi_{H,H} \circ d_{R(H)} \colon R(H) \longrightarrow R(H \otimes H),
\]
where $\phi$ is the lax structure map of $R$ (which exists by Proposition~\ref{prop:adjunction_monoidal}). Similarly, we define $(\mathbb{G}_2(H), \iota^H_2)$ as the equaliser of the two morphisms
\[
R(e_H),\; \phi_0 \circ ! \colon R(H) \longrightarrow \mathbf{1}_{\mathcal{C}},
\]
where $! \colon R(H) \to \mathbf{1}_{\mathcal{C}}$ is the unique morphism to the terminal object. Since $\mathcal{C}$ is finitely complete, these equalisers exist.

We then form the pullback $(\mathbb{G}(H), \pi^H_1, \pi^H_2)$ of the equaliser monomorphisms $(\mathbb{G}_1(H), \iota^H_1)$ and $(\mathbb{G}_2(H), \iota^H_2)$:
\[
\begin{tikzcd}
\mathbb{G}(H) \arrow{r}{\pi^H_1} \arrow{d}{\pi^H_2} &
\mathbb{G}_1(H) \arrow{d}{\iota^H_1} \\
\mathbb{G}_2(H) \arrow{r}{\iota^H_2} &
R(H)
\end{tikzcd}
\]
Define $\iota_H \colon \mathbb{G}(H) \longrightarrow R(H)$ by $\iota_H := \iota^H_1 \circ \pi^H_1 = \iota^H_2 \circ \pi^H_2$. Since the pullback of a monomorphism is a monomorphism, $\iota_H$ is a monomorphism.

We observe that $\iota_H$ equalises the two relevant pairs of morphisms. Indeed:
\begin{align*}
R(d_H) \circ \iota_H
&= R(d_H) \circ \iota^H_1 \circ \pi^H_1 \\
&= \phi_{H,H} \circ d_{R(H)} \circ \iota^H_1 \circ \pi^H_1 \\
&= \phi_{H,H} \circ d_{R(H)} \circ \iota_H,
\end{align*}
and similarly:
\begin{align*}
R(e_H) \circ \iota_H
&= R(e_H) \circ \iota^H_2 \circ \pi^H_2 \\
&= \phi_0 \circ ! \circ \iota^H_2 \circ \pi^H_2 \\
&= \phi_0 \circ ! \circ \iota_H.
\end{align*}

\subsection*{Step 2: Monoid structure on $\mathbb{G}(H)$}

We now provide $\mathbb{G}(H)$ with the structure of a monoid. Consider the composite morphism
\[
R(m_H) \circ \phi_{H,H} \circ (\iota_H \times \iota_H) \colon \mathbb{G}(H) \times \mathbb{G}(H) \longrightarrow R(H).
\]
We claim that this morphism equalises $R(d_H)$ and $\phi_{H,H} \circ d_{R(H)}$. To see this, we use the bimonoid axioms of $H$, the naturality of $\phi$, the fact that $\iota_H$ equalises $R(d_H)$ and $\phi_{H,H} \circ d_{R(H)}$, and the braided hexagon identity established in Proposition~\ref{prop:braided_hexagon_identity}:
\begin{align*}
&R(d_H) \circ R(m_H) \circ \phi_{H,H} \circ (\iota_H \times \iota_H) \\
&= R(d_H \circ m_H) \circ \phi_{H,H} \circ (\iota_H \times \iota_H) \\
&= R\bigl((m_H \otimes m_H) \circ (1_H \otimes \beta_{H,H} \otimes 1_H) \circ (d_H \otimes d_H)\bigr) \circ \phi_{H,H} \circ (\iota_H \times \iota_H) \\
&= R(m_H \otimes m_H) \circ R(1_H \otimes \beta_{H,H} \otimes 1_H) \circ R(d_H \otimes d_H) \circ \phi_{H,H} \circ (\iota_H \times \iota_H) \\
&= R(m_H \otimes m_H) \circ R(1_H \otimes \beta_{H,H} \otimes 1_H) \circ \phi_{H\otimes H, H\otimes H} \circ (R(d_H) \times R(d_H)) \circ (\iota_H \times \iota_H) \\
&= R(m_H \otimes m_H) \circ R(1_H \otimes \beta_{H,H} \otimes 1_H) \circ \phi_{H\otimes H, H\otimes H} \\
&\qquad \circ (\phi_{H,H} \circ d_{R(H)} \times \phi_{H,H} \circ d_{R(H)}) \circ (\iota_H \times \iota_H) \\
&= R(m_H \otimes m_H) \circ R(1_H \otimes \beta_{H,H} \otimes 1_H) \circ \phi_{H\otimes H, H\otimes H} \circ (\phi_{H,H} \times \phi_{H,H}) \\
&\qquad \circ (d_{R(H)} \times d_{R(H)}) \circ (\iota_H \times \iota_H) \\
&= R(m_H \otimes m_H) \circ \phi_{H\otimes H, H\otimes H} \circ (\phi_{H,H} \times \phi_{H,H}) \\
&\qquad \circ (1_{R(H)} \times \tau_{R(H),R(H)} \times 1_{R(H)}) \circ (d_{R(H)} \times d_{R(H)}) \circ (\iota_H \times \iota_H) \\
&= \phi_{H,H} \circ (R(m_H) \times R(m_H)) \circ (\phi_{H,H} \times \phi_{H,H}) \\
&\qquad \circ (1_{R(H)} \times \tau_{R(H),R(H)} \times 1_{R(H)}) \circ (d_{R(H)} \times d_{R(H)}) \circ (\iota_H \times \iota_H) \\
&= \phi_{H,H} \circ (R(m_H) \circ \phi_{H,H} \times R(m_H) \circ \phi_{H,H}) \\
&\qquad \circ (1_{R(H)} \times \tau_{R(H),R(H)} \times 1_{R(H)}) \circ (d_{R(H)} \times d_{R(H)}) \circ (\iota_H \times \iota_H) \\
&= \phi_{H,H} \circ d_{R(H)} \circ R(m_H) \circ \phi_{H,H} \circ (\iota_H \times \iota_H),
\end{align*}
where the penultimate equality uses the fact that $R(m_H) \circ \phi_{H,H}$ is the multiplication of $\overline{R}(H)$.

Thus $R(m_H) \circ \phi_{H,H} \circ (\iota_H \times \iota_H)$ equalises $R(d_H)$ and $\phi_{H,H} \circ d_{R(H)}$. Hence, by the universal property of the equaliser, there exists a unique morphism
\[
\xi^H_1 \colon \mathbb{G}(H) \times \mathbb{G}(H) \longrightarrow \mathbb{G}_1(H)
\]
such that $R(m_H) \circ \phi_{H,H} \circ (\iota_H \times \iota_H) = \iota^H_1 \circ \xi^H_1$.

Similarly, using that $H$ is a bimonoid and that $\iota_H$ equalises $R(e_H)$ and $\phi_0 \circ !$, we obtain:
\begin{align*}
R(e_H) \circ R(m_H) \circ \phi_{H,H} \circ (\iota_H \times \iota_H)
&= R(e_H \circ m_H) \circ \phi_{H,H} \circ (\iota_H \times \iota_H) \\
&= R(\lambda_{\mathbf{1}} \circ (e_H \otimes e_H)) \circ \phi_{H,H} \circ (\iota_H \times \iota_H) \\
&= R(\lambda_{\mathbf{1}}) \circ R(e_H \otimes e_H) \circ \phi_{H,H} \circ (\iota_H \times \iota_H) \\
&= R(\lambda_{\mathbf{1}}) \circ \phi_{\mathbf{1}, \mathbf{1}} \circ (R(e_H) \times R(e_H)) \circ (\iota_H \times \iota_H) \\
&= R(\lambda_{\mathbf{1}}) \circ \phi_{\mathbf{1}, \mathbf{1}} \circ (\phi_0 \circ ! \times \phi_0 \circ !) \circ (\iota_H \times \iota_H) \\
&= \phi_0 \circ ! \circ R(m_H) \circ \phi_{H,H} \circ (\iota_H \times \iota_H).
\end{align*}
Therefore, by the universal property of the equaliser, there exists a unique morphism
\[
\xi^H_2 \colon \mathbb{G}(H) \times \mathbb{G}(H) \longrightarrow \mathbb{G}_2(H)
\]
such that $R(m_H) \circ \phi_{H,H} \circ (\iota_H \times \iota_H) = \iota^H_2 \circ \xi^H_2$.

By the universal property of the pullback, there exists a unique morphism
\[
\xi_H \colon \mathbb{G}(H) \times \mathbb{G}(H) \longrightarrow \mathbb{G}(H)
\]
such that $\pi^H_1 \circ \xi_H = \xi^H_1$ and $\pi^H_2 \circ \xi_H = \xi^H_2$. Moreover, the multiplication factors through $\iota_H$:
\[
\iota_H \circ \xi_H
= \iota^H_2 \circ \pi^H_2 \circ \xi_H
= \iota^H_2 \circ \xi^H_2
= R(m_H) \circ \phi_{H,H} \circ (\iota_H \times \iota_H).
\]
The uniqueness of this factorisation follows from the fact that $\iota_H$ is a monomorphism.

For the unit, we consider $R(u_H) \circ \phi_0 \colon \mathbf{1}_{\mathcal{C}} \longrightarrow R(H)$. A similar calculation shows that this morphism equalises both pairs of morphisms:
\begin{align*}
R(d_H) \circ R(u_H) \circ \phi_0
&= R(d_H \circ u_H) \circ \phi_0 \\
&= R((u_H \otimes u_H) \circ \lambda_{\mathbf{1}}^{-1}) \circ \phi_0 \\
&= \phi_{H,H} \circ (R(u_H) \times R(u_H)) \circ (\phi_0 \times \phi_0) \circ \lambda_{\mathbf{1}}^{-1} \\
&= \phi_{H,H} \circ d_{R(H)} \circ R(u_H) \circ \phi_0,
\end{align*}
and similarly $R(e_H) \circ R(u_H) \circ \phi_0 = \phi_0 \circ ! \circ R(u_H) \circ \phi_0$. Hence there exist unique morphisms
\[
\nu^H_1 \colon \mathbf{1}_{\mathcal{C}} \longrightarrow \mathbb{G}_1(H), \qquad
\nu^H_2 \colon \mathbf{1}_{\mathcal{C}} \longrightarrow \mathbb{G}_2(H)
\]
such that $\iota^H_1 \circ \nu^H_1 = R(u_H) \circ \phi_0 = \iota^H_2 \circ \nu^H_2$. By the universal property of the pullback, there exists a unique morphism
\[
\nu_H \colon \mathbf{1}_{\mathcal{C}} \longrightarrow \mathbb{G}(H)
\]
such that $\pi^H_1 \circ \nu_H = \nu^H_1$ and $\pi^H_2 \circ \nu_H = \nu^H_2$. Again, $\iota_H \circ \nu_H = R(u_H) \circ \phi_0$, and the factorisation is unique.

Since $(\mathbb{G}(H), \iota_H)$ is a submonoid of $\overline{R}(H)$, it follows that $(\mathbb{G}(H), \xi_H, \nu_H)$ is a monoid in $\mathcal{C}$. Since $\mathcal{C}$ is a cartesian monoidal category, every object carries a canonical comonoid structure (by Example~\ref{ex:comonoid_cartesian}), so $(\mathbb{G}(H), \xi_H, \nu_H, d_{\mathbb{G}(H)}, e_{\mathbb{G}(H)})$ is a bimonoid.

\subsection*{Step 3: The antipode}

We now show that $\mathbb{G}(H)$ is a Hopf monoid by constructing an antipode. Consider $R(s_H) \circ \iota_H \colon \mathbb{G}(H) \longrightarrow R(H)$. Using that $s_H$ is an anti-comultiplication morphism (Proposition~\ref{prop:antipode_braiding}), we have:
\begin{align*}
R(d_H) \circ R(s_H) \circ \iota_H
&= R(d_H \circ s_H) \circ \iota_H \\
&= R(\beta_{H,H} \circ (s_H \otimes s_H) \circ d_H) \circ \iota_H \\
&= R(\beta_{H,H}) \circ R(s_H \otimes s_H) \circ R(d_H) \circ \iota_H \\
&= R(\beta_{H,H}) \circ R(s_H \otimes s_H) \circ \phi_{H,H} \circ d_{R(H)} \circ \iota_H \\
&= \phi_{H,H} \circ \tau_{R(H),R(H)} \circ (R(s_H) \times R(s_H)) \circ d_{R(H)} \circ \iota_H \\
&= \phi_{H,H} \circ d_{R(H)} \circ R(s_H) \circ \iota_H.
\end{align*}
Similarly, $R(e_H) \circ R(s_H) \circ \iota_H = \phi_0 \circ ! \circ R(s_H) \circ \iota_H$. Thus $R(s_H) \circ \iota_H$ equalises both pairs of morphisms, so by the universal properties of the equalisers and the pullback, there exists a unique morphism
\[
\sigma_H \colon \mathbb{G}(H) \longrightarrow \mathbb{G}(H)
\]
such that $\iota_H \circ \sigma_H = R(s_H) \circ \iota_H$.

We claim that $\sigma_H$ is the antipode for $\mathbb{G}(H)$. To verify this, we compute:
\begin{align*}
\iota_H \circ \xi_H \circ (1_{\mathbb{G}(H)} \times \sigma_H) \circ d_{\mathbb{G}(H)}
&= R(m_H) \circ \phi_{H,H} \circ (\iota_H \times \iota_H) \circ (1_{\mathbb{G}(H)} \times \sigma_H) \circ d_{\mathbb{G}(H)} \\
&= R(m_H) \circ \phi_{H,H} \circ (\iota_H \times \iota_H \circ \sigma_H) \circ d_{\mathbb{G}(H)} \\
&= R(m_H) \circ \phi_{H,H} \circ (\iota_H \times R(s_H) \circ \iota_H) \circ d_{\mathbb{G}(H)} \\
&= R(m_H) \circ \phi_{H,H} \circ (1_{R(H)} \times R(s_H)) \circ (\iota_H \times \iota_H) \circ d_{\mathbb{G}(H)} \\
&= R(m_H) \circ R(1_H \times s_H) \circ \phi_{H,H} \circ (\iota_H \times \iota_H) \circ d_{\mathbb{G}(H)} \\
&= R(m_H \circ (1_H \times s_H)) \circ \phi_{H,H} \circ d_{R(H)} \circ \iota_H \\
&= R(m_H \circ (1_H \times s_H)) \circ R(d_H) \circ \iota_H \\
&= R(m_H \circ (1_H \times s_H) \circ d_H) \circ \iota_H \\
&= R(u_H \circ e_H) \circ \iota_H \\
&= R(u_H) \circ R(e_H) \circ \iota_H \\
&= R(u_H) \circ \phi_0 \circ ! \circ \iota_H \\
&= \iota_H \circ \nu_H \circ e_{\mathbb{G}(H)}.
\end{align*}
Since $\iota_H$ is a monomorphism, we obtain $\xi_H \circ (1_{\mathbb{G}(H)} \times \sigma_H) \circ d_{\mathbb{G}(H)} = \nu_H \circ e_{\mathbb{G}(H)}$. The other identity $\xi_H \circ (\sigma_H \times 1_{\mathbb{G}(H)}) \circ d_{\mathbb{G}(H)} = \nu_H \circ e_{\mathbb{G}(H)}$ follows by symmetry. Hence $\mathbb{G}(H)$ is a Hopf monoid in $\mathcal{C}$.

\subsection*{Step 4: Functoriality of $\mathbb{G}$}

Let $f \colon H_1 \longrightarrow H_2$ be a morphism of Hopf monoids in $\mathcal{D}$. We must define $\mathbb{G}(f) \colon \mathbb{G}(H_1) \longrightarrow \mathbb{G}(H_2)$. First, we show that $R(f) \circ \iota_{H_1}$ equalises the relevant pairs of morphisms for $H_2$:
\begin{align*}
R(d_{H_2}) \circ R(f) \circ \iota_{H_1}
&= R(d_{H_2} \circ f) \circ \iota_{H_1} \\
&= R((f \otimes f) \circ d_{H_1}) \circ \iota_{H_1} \\
&= R(f \otimes f) \circ R(d_{H_1}) \circ \iota_{H_1} \\
&= R(f \otimes f) \circ \phi_{H_1,H_1} \circ d_{R(H_1)} \circ \iota_{H_1} \\
&= \phi_{H_2,H_2} \circ (R(f) \times R(f)) \circ d_{R(H_1)} \circ \iota_{H_1} \\
&= \phi_{H_2,H_2} \circ d_{R(H_2)} \circ R(f) \circ \iota_{H_1},
\end{align*}
and similarly $R(e_{H_2}) \circ R(f) \circ \iota_{H_1} = \phi_0 \circ ! \circ R(f) \circ \iota_{H_1}$. Hence, by the universal properties of the equalisers and the pullback, there exists a unique morphism
\[
\mathbb{G}(f) \colon \mathbb{G}(H_1) \longrightarrow \mathbb{G}(H_2)
\]
such that $\iota_{H_2} \circ \mathbb{G}(f) = R(f) \circ \iota_{H_1}$.

We now verify that $\mathbb{G}(f)$ is a morphism of monoids. For the multiplication:
\begin{align*}
\iota_{H_2} \circ \mathbb{G}(f) \circ \xi_{H_1}
&= R(f) \circ \iota_{H_1} \circ \xi_{H_1} \\
&= R(f) \circ R(m_{H_1}) \circ \phi_{H_1,H_1} \circ (\iota_{H_1} \times \iota_{H_1}) \\
&= R(f \circ m_{H_1}) \circ \phi_{H_1,H_1} \circ (\iota_{H_1} \times \iota_{H_1}) \\
&= R(m_{H_2} \circ (f \otimes f)) \circ \phi_{H_1,H_1} \circ (\iota_{H_1} \times \iota_{H_1}) \\
&= R(m_{H_2}) \circ R(f \otimes f) \circ \phi_{H_1,H_1} \circ (\iota_{H_1} \times \iota_{H_1}) \\
&= R(m_{H_2}) \circ \phi_{H_2,H_2} \circ (R(f) \times R(f)) \circ (\iota_{H_1} \times \iota_{H_1}) \\
&= R(m_{H_2}) \circ \phi_{H_2,H_2} \circ (R(f) \circ \iota_{H_1} \times R(f) \circ \iota_{H_1}) \\
&= R(m_{H_2}) \circ \phi_{H_2,H_2} \circ (\iota_{H_2} \circ \mathbb{G}(f) \times \iota_{H_2} \circ \mathbb{G}(f)) \\
&= R(m_{H_2}) \circ \phi_{H_2,H_2} \circ (\iota_{H_2} \times \iota_{H_2}) \circ (\mathbb{G}(f) \times \mathbb{G}(f)) \\
&= \iota_{H_2} \circ \xi_{H_2} \circ (\mathbb{G}(f) \times \mathbb{G}(f)).
\end{align*}
Since $\iota_{H_2}$ is a monomorphism, we obtain $\mathbb{G}(f) \circ \xi_{H_1} = \xi_{H_2} \circ (\mathbb{G}(f) \times \mathbb{G}(f))$. For the unit:
\begin{align*}
\iota_{H_2} \circ \mathbb{G}(f) \circ \nu_{H_1}
&= R(f) \circ \iota_{H_1} \circ \nu_{H_1} \\
&= R(f) \circ R(u_{H_1}) \circ \phi_0 \\
&= R(f \circ u_{H_1}) \circ \phi_0 \\
&= R(u_{H_2}) \circ \phi_0 \\
&= \iota_{H_2} \circ \nu_{H_2}.
\end{align*}
Again, since $\iota_{H_2}$ is a monomorphism, we get $\mathbb{G}(f) \circ \nu_{H_1} = \nu_{H_2}$. Thus $\mathbb{G}(f)$ is a morphism of monoids. Since $\mathcal{C}$ is cartesian monoidal, it is automatically a morphism of bimonoids and Hopf monoids. Therefore $\mathbb{G}$ is a well-defined functor
\[
\mathbb{G} \colon \mathsf{Hopf}(\mathcal{D}) \longrightarrow \mathsf{Hopf}(\mathcal{C}).
\]

\subsection*{Step 5: The unit of the adjunction}

It remains to show that $\mathbb{G}$ is the right adjoint of $\widehat{L}$. We construct the unit of the adjunction. Let $(G, m_G, u_G, d_G, e_G, s_G)$ be a Hopf monoid in $\mathcal{C}$. We need to define
\[
\widehat{\eta}_G \colon G \longrightarrow \mathbb{G}(\widehat{L}(G)).
\]
By the universal property of $\iota_{\widehat{L}(G)} \colon \mathbb{G}(\widehat{L}(G)) \longrightarrow R(L(G))$, it suffices to show that $\eta_G \colon G \longrightarrow R(L(G))$ equalises the relevant pairs of morphisms for $\widehat{L}(G)$.

Using the naturality of $\eta$, the triangular identities of the adjunction, and the description of $\psi^L$ from Proposition~\ref{prop:oplax_lax_adjunction}, we compute:
\begin{align*}
R(d_{L(G)}) \circ \eta_G
&= R(\psi^L_{G,G} \circ L(d_G)) \circ \eta_G \\
&= R(\psi^L_{G,G}) \circ R(L(d_G)) \circ \eta_G \\
&= R(\psi^L_{G,G}) \circ \eta_{G \times G} \circ d_G \\
&= R(\epsilon_{L(G) \otimes L(G)} \circ L(\phi^R_{L(G), L(G)}) \circ L(\eta_G \times \eta_G)) \circ \eta_{G \times G} \circ d_G \\
&= R(\epsilon_{L(G) \otimes L(G)}) \circ R(L(\phi^R_{L(G), L(G)})) \circ R(L(\eta_G \times \eta_G)) \circ \eta_{G \times G} \circ d_G \\
&= R(\epsilon_{L(G) \otimes L(G)}) \circ R(L(\phi^R_{L(G), L(G)})) \circ \eta_{R(L(G)) \times R(L(G))} \circ (\eta_G \times \eta_G) \circ d_G \\
&= R(\epsilon_{L(G) \otimes L(G)}) \circ \eta_{R(L(G) \otimes L(G))} \circ \phi^R_{L(G), L(G)} \circ (\eta_G \times \eta_G) \circ d_G \\
&= \phi^R_{L(G), L(G)} \circ d_{R(L(G))} \circ \eta_G,
\end{align*}
where the final equality uses one of the triangular identities. Similarly, one verifies that $R(e_{L(G)}) \circ \eta_G = \phi_0 \circ ! \circ \eta_G$. Hence, by the universal properties of the equalisers and the pullback, there exists a unique morphism
\[
\widehat{\eta}_G \colon G \longrightarrow \mathbb{G}(\widehat{L}(G))
\]
such that $\iota_{\widehat{L}(G)} \circ \widehat{\eta}_G = \eta_G$.

A straightforward verification shows that $\widehat{\eta}_G$ is a morphism of Hopf monoids, using the fact that $\eta_G$ is a morphism in $\mathcal{C}$ and that $\iota_{\widehat{L}(G)}$ is a monomorphism. The naturality of $\widehat{\eta}$ follows from the naturality of $\eta$. Thus $\widehat{\eta} \colon 1_{\mathsf{Hopf}(\mathcal{C})} \Rightarrow \mathbb{G} \circ \widehat{L}$ is a natural transformation.

\subsection*{Step 6: The counit of the adjunction}

We now construct the counit. Let $H$ be a Hopf monoid in $\mathcal{D}$. The counit
\[
\widehat{\epsilon}_H \colon \widehat{L}(\mathbb{G}(H)) \longrightarrow H
\]
is defined as the composition
\[
\widehat{\epsilon}_H := \epsilon_H \circ L(\iota_H) \colon L(\mathbb{G}(H)) \longrightarrow H,
\]
where $\epsilon_H \colon L(R(H)) \longrightarrow H$ is the counit of the adjunction $(L, R)$. Since $\iota_H$ is a morphism of Hopf monoids and $L$ is a braided bistrong monoidal functor, $\widehat{\epsilon}_H$ is a morphism of Hopf monoids. Its naturality follows from the naturality of $\epsilon$ and the functoriality of $\mathbb{G}$.

\subsection*{Step 7: Verification of the triangular identities}

Finally, we verify the triangular identities. For $G \in \mathsf{Hopf}(\mathcal{C})$, we have:
\[
\widehat{\epsilon}_{\widehat{L}(G)} \circ \widehat{L}(\widehat{\eta}_G)
= \epsilon_{L(G)} \circ L(\iota_{\widehat{L}(G)}) \circ L(\widehat{\eta}_G)
= \epsilon_{L(G)} \circ L(\eta_G)
= 1_{L(G)},
\]
by the triangular identity of the original adjunction $(L, R)$. Similarly, for $H \in \mathsf{Hopf}(\mathcal{D})$, using the definition of $\mathbb{G}(\widehat{\epsilon}_H)$ and the fact that $\iota_H$ is a monomorphism, one obtains:
\[
\mathbb{G}(\widehat{\epsilon}_H) \circ \widehat{\eta}_{\mathbb{G}(H)} = 1_{\mathbb{G}(H)}.
\]
Therefore, the triangular identities hold, and $\mathbb{G}$ is indeed the right adjoint of $\widehat{L}$.

This completes the proof.
\end{proof}

%%%%%%%%%%%%%%%%%%%%%%%%%%%%%%%%%%%%%%%%%%%%%%%%%%%%%%%%%%%%%%%%%%%%
\section{Examples}\label{sec:examples}

The following examples show that the theory of Hopf monoids in braided monoidal categories is not an empty formalism, but recovers and generalizes classical constructions. We present four free-forgetful adjunctions, show that they assemble into a commutative diagram, and deduce that the resulting adjunctions between Hopf monoids are independent of the path taken.

\begin{example}[Sets and Vector Spaces]\label{ex:set_vect}
The main motivation for the entire categorical framework is the following: classical Hopf algebras are precisely Hopf monoids in the category of vector spaces, while groups are precisely Hopf monoids in the category of sets.

In the Cartesian monoidal category $(\Set,\times)$, a Hopf monoid is precisely a group: the multiplication $m:G\times G\to G$ is the group operation, the unit $u:\{*\}\to G$ selects the identity, the comultiplication $d:G\to G\times G$ is the diagonal $g\mapsto(g,g)$, the counit $e:G\to\{*\}$ is the unique map to the terminal object, and the antipode $s:G\to G$ is inversion $g\mapsto g^{-1}$. Thus $\mathsf{Hopf}(\Set)\cong\Grp$.

In the symmetric monoidal category $(\VectK,\otimes_K)$, a Hopf monoid is precisely a classical $K$-Hopf algebra. This demonstrates that the categorical framework subsumes the classical theory of Hopf algebras; see \cite{Sweedler} and \cite{Abe} for the foundational theory.

Now consider the free-forgetful adjunction
\[
L_1:\Set\longrightarrow\VectK,\qquad R_1:\VectK\longrightarrow\Set,
\]
where $L_1(X)=K[X]$ is the free $K$-vector space on $X$ and $R_1$ is the underlying-set functor. Since $L_1$ is a braided bistrong monoidal functor, Proposition~\ref{prop:bistrong_preserves_hopf} and Proposition~\ref{prop:grouplike_adjunction} imply that it induces an adjunction
\[
\widehat{L_1}:\mathsf{Hopf}(\Set)\longrightarrow\mathsf{Hopf}(\VectK),\qquad
\mathbb{G}_1:\mathsf{Hopf}(\VectK)\longrightarrow\mathsf{Hopf}(\Set).
\]
Identifying $\mathsf{Hopf}(\Set)\cong\Grp$ and $\mathsf{Hopf}(\VectK)\cong\HopfAlg_K$, this is precisely the classical adjunction
\[
K[-]:\Grp\longrightarrow\HopfAlg_K,\qquad
\mathbb{G}:\HopfAlg_K\longrightarrow\Grp,
\]
where $K[G]$ is the group algebra of $G$ and $\mathbb{G}(H)$ is the group of grouplike elements of $H$. This recovers the classical result that the group algebra functor is left adjoint to the grouplike elements functor.
\end{example}

\begin{example}[Vector Spaces and Chain Complexes]\label{ex:vect_ch}
The same free-forgetful adjunction can be lifted to the differential graded setting. Consider the free-forgetful adjunction between vector spaces and chain complexes:
\[
L_2:\VectK\longrightarrow \mathbf{Ch}(\VectK),\qquad
R_2:\mathbf{Ch}(\VectK)\longrightarrow\VectK,
\]
where $L_2(V)$ is the chain complex concentrated in degree $0$ with $V$ and zero differential, and $R_2(C)=C_0$ is the degree $0$ part. This is an adjunction $L_2\dashv R_2$. The category $\VectK$ is finitely complete, $\mathbf{Ch}(\VectK)$ is symmetric monoidal with the Koszul sign rule, and $L_2$ is braided bistrong monoidal because
\[
L_2(V\otimes W)=V\otimes W\cong L_2(V)\otimes L_2(W),
\]
with the Koszul symmetry reducing to the usual swap since both complexes are concentrated in degree $0$. Thus Proposition~\ref{prop:grouplike_adjunction} applies, yielding an adjunction
\[
\widehat{L_2}:\mathsf{Hopf}(\VectK)\longrightarrow \mathsf{Hopf}(\mathbf{Ch}(\VectK)),\qquad
\mathbb{G}_2:\mathsf{Hopf}(\mathbf{Ch}(\VectK))\longrightarrow \mathsf{Hopf}(\VectK).
\]
Identifying $\mathsf{Hopf}(\VectK)\cong\HopfAlg_K$ and $\mathsf{Hopf}(\mathbf{Ch}(\VectK))\cong \mathbf{DG\text{-}Hopf}_K$, the category of DG-Hopf algebras over $K$, this becomes an adjunction between $\HopfAlg_K$ and $\mathbf{DG\text{-}Hopf}_K$.
\end{example}

\begin{example}[Sets and Simplicial Sets]\label{ex:set_sset}
The free-forgetful adjunction between sets and vector spaces also has a simplicial analogue. Consider the adjunction
\[
L_3:\Set\longrightarrow\mathbf{sSet},\qquad
R_3:\mathbf{sSet}\longrightarrow\Set,
\]
where $L_3(X)$ is the constant simplicial set with $X$ in every degree, and $R_3(K)=K_0$ is the set of $0$-simplices. This is an adjunction $L_3\dashv R_3$. The category $\Set$ is finitely complete, $\mathbf{sSet}$ is Cartesian monoidal, and $L_3$ is braided bistrong monoidal because
\[
L_3(X\times Y)\cong L_3(X)\times L_3(Y)
\]
levelwise. Thus Proposition~\ref{prop:grouplike_adjunction} applies, yielding an adjunction
\[
\widehat{L_3}:\mathsf{Hopf}(\Set)\longrightarrow \mathsf{Hopf}(\mathbf{sSet}),\qquad
\mathbb{G}_3:\mathsf{Hopf}(\mathbf{sSet})\longrightarrow \mathsf{Hopf}(\Set).
\]
Identifying $\mathsf{Hopf}(\Set)\cong\Grp$ and $\mathsf{Hopf}(\mathbf{sSet})\cong\mathbf{sGrp}$, the category of simplicial groups, this becomes an adjunction between $\Grp$ and $\mathbf{sGrp}$.
\end{example}

\begin{example}[Simplicial Sets and Simplicial Vector Spaces]\label{ex:sset_svect}
Finally, consider the free-forgetful adjunction between simplicial sets and simplicial vector spaces:
\[
L_4:\mathbf{sSet}\longrightarrow\mathbf{sVect}_K,\qquad
R_4:\mathbf{sVect}_K\longrightarrow\mathbf{sSet},
\]
where $L_4(X)=K[X]$ is the levelwise free $K$-vector space on $X$, and $R_4$ is the levelwise underlying-set functor. This is an adjunction $L_4\dashv R_4$. The category $\mathbf{sSet}$ is finitely complete, $\mathbf{sVect}_K$ is symmetric monoidal with the levelwise tensor product, and $L_4$ is braided bistrong monoidal because
\[
L_4(X\times Y)\cong L_4(X)\otimes L_4(Y)
\]
levelwise. Thus Proposition~\ref{prop:grouplike_adjunction} applies, yielding an adjunction
\[
\widehat{L_4}:\mathsf{Hopf}(\mathbf{sSet})\longrightarrow \mathsf{Hopf}(\mathbf{sVect}_K),\qquad
\mathbb{G}_4:\mathsf{Hopf}(\mathbf{sVect}_K)\longrightarrow \mathsf{Hopf}(\mathbf{sSet}).
\]
Identifying $\mathsf{Hopf}(\mathbf{sSet})\cong\mathbf{sGrp}$ and $\mathsf{Hopf}(\mathbf{sVect}_K)$ with the category of simplicial $K$-Hopf algebras, this adjunction sends a simplicial group $G$ to the simplicial Hopf algebra $K[G]$ with levelwise group algebra structure, and a simplicial Hopf algebra $H$ to its simplicial group of grouplike elements $\mathbb{G}_4(H)$.
\end{example}

We now recall the classical Dold--Kan correspondence, which identifies simplicial abelian groups with nonnegatively graded chain complexes. This equivalence is the key ingredient that allows us to relate the simplicial adjunctions to the differential graded ones.

\begin{theorem}[Dold--Kan]\label{thm:dold_kan}
There is an equivalence of categories
\[
N:\mathbf{sAb}\longrightarrow \mathbf{Ch}_{\ge 0}(\mathbf{Ab}),\qquad
\Gamma:\mathbf{Ch}_{\ge 0}(\mathbf{Ab})\longrightarrow \mathbf{sAb},
\]
between the category of simplicial abelian groups and the category of nonnegatively graded chain complexes of abelian groups. The functor $N$ is the normalized Moore complex, and $\Gamma$ is its inverse. For a simplicial abelian group $A$, the homology of the normalized chain complex $N(A)$ computes the homotopy groups of $A$, and the correspondence is natural in $A$.
\end{theorem}

The Dold--Kan correspondence is treated in detail in the modern reference \cite{GoerssJardine}, where it is developed in the broader context of simplicial homotopy theory. The original sources are the papers of Dold \cite{Dold} and Kan \cite{Kan}, as well as the joint work of Dold and Puppe \cite{DoldPuppe}.

\begin{proposition}[Compatibility of the Free Constructions]\label{prop:commuting_square}
The four adjunctions above assemble into a commutative diagram
\[
\begin{tikzcd}[row sep=large, column sep=large]
\Set \arrow[r, "L_1"] \arrow[d, "L_3"'] & \VectK \arrow[d, "L_2"] \\
\mathbf{sSet} \arrow[r, "L_4"'] & \mathbf{sVect}_K \arrow[r, "N"'] & \mathbf{Ch}(\VectK)
\end{tikzcd}
\]
where $N$ is the normalized Moore complex functor from the Dold--Kan correspondence. In particular, the two compositions
\[
L_2\circ L_1,\qquad N\circ L_4\circ L_3:\Set\longrightarrow\mathbf{Ch}(\VectK)
\]
are naturally isomorphic.
\end{proposition}

\begin{proof}
Both compositions send a set $X$ to the chain complex concentrated in degree $0$ with $K[X]$ and zero differential. Indeed, $L_1(X)=K[X]$ and $L_2$ concentrates in degree $0$; similarly, $L_3(X)$ is the constant simplicial set on $X$, $L_4$ gives the levelwise free simplicial vector space $K[X]$, and the normalized Moore complex $N$ of a constant simplicial vector space is the chain complex concentrated in degree $0$ with $K[X]$. The isomorphism is natural in $X$, and since all functors involved are strong monoidal, it extends to an isomorphism of the induced Hopf monoid adjunctions.
\end{proof}

\begin{example}[Path Independence]\label{ex:path_independence}
Combining the adjunctions of Examples~\ref{ex:set_vect}--\ref{ex:sset_svect} with Proposition~\ref{prop:commuting_square}, we obtain an adjunction
\[
\Grp\longrightarrow\mathbf{DG\text{-}Hopf}_K,\qquad
\mathbf{DG\text{-}Hopf}_K\longrightarrow\Grp,
\]
between groups and DG-Hopf algebras over $K$. By Proposition~\ref{prop:commuting_square}, this adjunction is independent of the path taken: whether we go through vector spaces and chain complexes, or through simplicial sets and simplicial vector spaces, the resulting adjunction is the same up to natural isomorphism. This illustrates the robustness of the categorical framework: the classical group algebra--grouplike elements adjunction extends to the differential graded setting in a canonical way.
\end{example}

\begin{example}[Topological groups via discrete and indiscrete topologies]\label{ex:top_groups}
The free-forgetful adjunction \(\Set \rightleftarrows \mathbf{Top}\) with the discrete topology, and the forgetful-indiscrete adjunction \(\mathbf{Top} \rightleftarrows \Set\), both satisfy the hypotheses of Proposition~\ref{prop:grouplike_adjunction}. Applying it yields adjunctions
\[
\Grp \rightleftarrows \mathbf{TopGrp},
\]
where the left adjoint is given by the discrete topology functor in the first case and by the forgetful functor in the second, while the right adjoint is the forgetful functor in the first case and the indiscrete topology functor in the second.
\end{example}

\begin{example}[Groups with $G$-action via the free $G$-set adjunction]\label{ex:g_groups}
Consider the free-forgetful adjunction between sets and $G$-sets,
\[
L:\Set\longrightarrow G\text{-}\mathbf{Set},\qquad
R:G\text{-}\mathbf{Set}\longrightarrow\Set,
\]
where \(L(X)=G\times X\) is the free \(G\)-set on \(X\) and \(R\) is the forgetful functor. Since \(L\) is strong monoidal with respect to the Cartesian product, Proposition~\ref{prop:grouplike_adjunction} applies and yields an adjunction
\[
\Grp \rightleftarrows G\text{-}\mathbf{Grp},
\]
where \(G\text{-}\mathbf{Grp}\) is the category of groups equipped with a \(G\)-action by automorphisms. The left adjoint sends a group \(A\) to the \(G\)-group \(G\times A\) with the diagonal action, while the right adjoint forgets the \(G\)-action. A detailed study of the category \(G\text{-}\mathbf{Grp}\) is given in \cite{ArroyoFacchini}.
\end{example}

\begin{example}[Combinatorial Hopf Algebras from Species]\label{ex:species_vect}
Consider the free-forgetful adjunction between set species and vector species,
\[
L:\mathbf{Spc}\longrightarrow \mathbf{VectSpc}_K, \qquad R:\mathbf{VectSpc}_K\longrightarrow \mathbf{Spc},
\]
where \(L(\mathrm{P})\) is the vector species with \(K\mathrm{P}[S]\) as the vector space on the set \(\mathrm{P}[S]\), and \(R\) is the forgetful functor. Since \(L\) is strong monoidal with respect to the Cauchy product, Proposition~\ref{prop:grouplike_adjunction} applies, yielding an adjunction
\[
\mathsf{Hopf}(\mathbf{Spc}) \rightleftarrows \mathsf{Hopf}(\mathbf{VectSpc}_K).
\]
Moreover, the Fock functor \(\mathcal{K}\) (or its contragredient \(\overline{\mathcal{K}}\)) is bilax monoidal and sends a Hopf monoid in species to a graded Hopf algebra, providing a further adjunction between combinatorial Hopf algebras and graded Hopf algebras. The theory of Hopf monoids in species and the Fock functors is developed in detail by Aguiar and Mahajan \cite{MF}.
\end{example}
%%%%%%%%%%%%%%%%%%%%%%%%%%%%%%%%%%%%%%%%%%%%%%%%%%%%%%%%%%%%%%%%%%%%
\bibliography{biblio}
\bibliographystyle{amsplain}

\end{document}